\documentclass[12pt]{amsart}
\usepackage[utf8]{inputenc}
\usepackage[T1]{fontenc}

\usepackage{fullpage, url, amssymb, bm}
\usepackage{enumitem, graphicx}
\usepackage{accents}
\usepackage{amsmath}
\usepackage{amsthm}

\usepackage{xcolor}
\definecolor{olive}{rgb}{0.3, 0.4, .1}
\definecolor{fore}{RGB}{249,242,215}
\definecolor{back}{RGB}{51,51,51}
\definecolor{title}{RGB}{255,0,90}
\definecolor{dgreen}{rgb}{0.,0.6,0.}
\definecolor{dblue}{rgb}{0.,0.,0.7}
\definecolor{gold}{rgb}{1.,0.84,0.}
\definecolor{JungleGreen}{cmyk}{0.99,0,0.52,0}
\definecolor{BlueGreen}{cmyk}{0.85,0,0.33,0}
\definecolor{Navy}{RGB}{0,0,128}
\definecolor{RawSienna}{cmyk}{0,0.72,1,0.45}
\definecolor{Magenta}{cmyk}{0,1,0,0}

\font\elevensc=cmcsc10 scaled\magstephalf

\font\teneu=eufm10 scaled\magstep1\font\seveneu=eufm7\font\fiveeu=eufm5

\newfam\eufam  \def\eu{\fam\eufam\teneu}
\textfont\eufam=\teneu \scriptfont\eufam=\seveneu
                              \scriptscriptfont\eufam=\fiveeu

\newcommand{\ha}[1]{{\hbox to#1pt{}}}
\newcommand{\hb}[1]{{\hbox to-#1pt{}}}

\newcommand{\itm}[2]{\begin{itemize}[leftmargin=#1pt]{#2}\end{itemize}}

\newcommand{\hla}{\hookleftarrow}
\newcommand{\hra}{\hookrightarrow}

\newcommand{\rsdp}{{\,\times\kern-3pt\lower-1pt%
\hbox{$\scriptscriptstyle|$\ha3}}}

\newcommand{\dwnr}[2]{{\ha{#1}\downarrow\hb2\lower-5pt\hbox{$\scriptstyle #2$}}}
\newcommand{\dwnl}[2]{{\lower-5pt\hbox{${\scriptstyle #2}$}\hb2\downarrow\ha{#1}}}

\newcommand{\hhb}[1]{\hbox to#1pt{}}
\newcommand{\hor}[1]{\smash
       {\mathop{{\lgrghtar}}\limits^{\lower2pt\hbox{$\scriptstyle #1$}}}}
\newcommand{\horr}[1]{\ha2\smash
      {\mathop{{\lglgrghtar}}\limits^{\lower2pt\hbox{$\scriptstyle #1$}}}\ha2}
\newcommand{\mpst}[1]{\,\smash
       {\mathop{{\mapsto}}\limits^{\lower2pt\hbox{$\scriptstyle #1$}}}\,}
\newcommand{\mpstt}[1]{\,\smash
       {\mathop{{\longmapsto}}\limits^{\lower2pt\hbox{$\scriptstyle #1$}}}\,}

\newcommand{\lgrghtar}{{\ha2{\relbar\joinrel\rightarrow}\ha2}}
\newcommand{\lglgrghtar}{{\ha1{\relbar\joinrel\relbar\joinrel\rightarrow}\ha1}}

\newcommand{\nmnm}[1]{{\scalebox{.85}[1.05]{\elevensc #1}}}

\newcommand{\plim}[1]{\hbox to14pt{\rm
lim\kern-14pt\lower4.5pt\hbox{$\scriptstyle\longleftarrow$}%
\kern-8pt\lower8.5pt\hbox{$\scriptstyle{#1}$}}\ha{3}}

\newcommand{\ilim}[1]{\hbox to14pt{\rm
lim\kern-15pt\lower4.5pt\hbox{$\scriptstyle\longrightarrow$}%
\kern-8pt\lower8.5pt\hbox{$\scriptstyle{#1}$}}\ha{3}}

\newcommand{\clD}{{\mathcal D}}

\newcommand{\clO}{{\mathcal O}}

\newcommand{\clT}{{\mathcal T}}

\newcommand{\clV}{{\mathcal V}}
\newcommand{\clW}{{\mathcal W}}
\newcommand{\clX}{{\mathcal X}}

\newcommand{\eua}{{\eu b}}

\newcommand{\eum}{{\eu m}}

\newcommand{\eup}{{\eu p}}

\newcommand{\euw}{{\eu w}}

\newcommand{\lvF}{{\mathbb F}}

\newcommand{\lvP}{{\mathbb P}}
\newcommand{\lvQ}{{\mathbb Q}}

\newcommand{\lvZ}{{\mathbb Z}}
\DeclareMathOperator{\cd}{cd}
\DeclareMathOperator{\chr}{char}
\DeclareMathOperator{\codim}{codim}

\DeclareMathOperator{\Hom}{Hom}

\DeclareMathOperator{\Quot}{Quot}
\DeclareMathOperator{\res}{res}

\DeclareMathOperator{\Spec}{Spec}

\DeclareMathOperator{\td}{td}

\newcommand{\Br}{{\rm Br}}

\newcommand{\bma}{\bm a\hb{7.5}\bm a}

\newcommand{\bmt}{\bm t\hb{5}\bm t}

\newcommand{\bmu}{\bm u\hb{8}\bm u}

\newcommand{\bmx}{\bm x\hb{8}\bm x}

\newcommand{\bpl}{\hbox{\rm\large(}}
\newcommand{\bpr}{\hbox{\rm\large)}}

\newcommand{\chm}{{\scriptscriptstyle\chi}}

\newcommand{\dts}{\hb1.\ha1.\ha1.\ha1}

\newcommand{\Fw}{F\hb1w}

\newcommand{\fps}[1]{[\hb{1.5}[#1]\hb{1.5}]}

\newcommand{\HHx}[3]{{\rm H}^{#1}\bpl#2,\lvZ/n(#3)\bpr}
\newcommand{\HHy}[2]{{\rm H}^{#1}\big(#2\big)}

\newcommand{\HH}{{\rm H}}
\newcommand{\hatvK}{{K_{\widehat v}}}

\newcommand{\hatkov}{{k_{1\widehat v}}}
\newcommand{\hatv}[1]{{#1_{\widehat v}}}

\newcommand{\hns}[2]{#1_#2}

\newcommand{\Kw}{K\hb{1}w}

\newcommand{\ko}{k_{\scalebox{.5}[.5]{\rm1}}}
\newcommand{\kov}{k_{{\scriptscriptstyle1}v}}
\newcommand{\kovh}{k_{{\scriptscriptstyle1}v}}

\newcommand{\LL}{{\Lambda}}

\newcommand{\lo}{l_{\scriptscriptstyle1}}

\newcommand{\lps}[1]{(\hb{1.5}(#1)\hb{1.5})}

\newcommand{\lx}{k_{\bm\epsilon}}

\newcommand{\oli}{\overline}

\newcommand{\Pff}[1]{\langle\hb2\langle #1\rangle\hb2\rangle}
\newcommand{\Pft}[1]{\langle\hb2\langle #1]\hb{.75}]}

\newcommand{\sps}{{\hbox{\ha1-\ha1}}}
\newcommand{\Th}{{\mathfrak{Th}}}

\newcommand{\tlK}{{\tilde K}}  

\newcommand{\tlX}{{\tilde X}}  

\newcommand{\tlw}{{\tilde w}}  
\newcommand{\tlx}{{\tilde x}}

\newcommand{\tms}{^{\hb1\times}\!}
\newcommand{\tttt}{{\theta}}

\newcommand{\Val}[1]{{\rm Val}_{#1}}

\newcommand{\veps}{\varepsilon}

\newcommand{\val}{\scalebox{.85}[1]{{\textnormal{\textbf{\textsf{val}}}}}}

\newcommand{\Q}{\mathbb{Q}}
\newcommand{\Z}{\mathbb{Z}}
\newcommand{\F}{\mathbb{F}}

\theoremstyle{plain}

\newtheorem{theorem}{Theorem}[section]
\newtheorem*{maintheo*}{Main Theorem}

\newtheorem{lemma}[theorem]{Lemma}
\newtheorem{proposition}[theorem]{Proposition}

\theoremstyle{definition}

\newtheorem{definition}[theorem]{Definition}

\newtheorem{example/fact}[theorem]{Example/Fact}
\newtheorem{fact}[theorem]{Fact}

\newtheorem{fact/definition}[theorem]{Fact/Definition}

\newtheorem{notations/remarks}[theorem]{Notations/Remarks}
\newtheorem{definitions/notations}[theorem]{Definitions/Notations}
\newtheorem{notations/facts}[theorem]{Notations/Facts}

\newtheorem{recipe}[theorem]{Recipe}
\newtheorem{remark/definition}[theorem]{Remark/Definition}
\newtheorem{definition/remarks}[theorem]{Definition/Remarks}
\newtheorem{remark}[theorem]{Remark}

\begin{document}

\title{
  Characterizing finitely generated fields\\ by a single field axiom}
\author{Philip Dittmann and Florian Pop}
\date{27 April 2023}

\maketitle

\begin{center}
\it dem Andenken an Peter Roquette gewidmet
\end{center}

\begin{abstract}
  We resolve the strong Elementary Equivalence versus Isomorphism
  Problem for finitely generated fields. That is, we show that for
  every field in this class there is a first-order sentence which
  characterizes this field within the class up to isomorphism. Our
  solution is conditional on resolution of singularities in
  characteristic two and unconditional in all other characteristics.
\end{abstract}

\section{Introduction}

First-order logic naturally applies to the study of fields. 
Consequently, it is of interest to investigate the expressive 
power of first-order logic in natural classes of fields.
This is well-understood in the cases of algebraically 
closed fields, real-closed fields and $p$-adically closed fields. 
Namely, every such field $K$ is elementary equivalent 
to its ``constant field'' $\kappa$, i.e., the relative algebraic 
closure of the prime field in $K$, and its first-order theory 
is decidable.

This article is concerned with fields which are at the centre of 
(birational) arithmetic geometry, namely the finitely generated 
fields $K$, which are the function fields of integral $\Z$-schemes 
of finite type. The \emph{Elementary Equivalence versus Isomorphism 
Problem,} for short EEIP, asks whether the elementary
theory $\Th(K)$ of a finitely generated field $K$ (always in 
the language of rings) encodes the isomorphism type 
of $K$ in the class of all finitely generated fields. This 
question goes back to the 1970s and seems to have first
been posed explicitly in \cite{P_EEI}, with the work of 
\nmnm{Rumely}~\cite{Rumely}, \nmnm{Duret} \cite{Du} 
and \nmnm{Pierce}~\cite{Pi} notable predecessors.

On the other hand, through the work of \nmnm{Rumely} 
\cite{Rumely}, much more than the EEIP is known for 
global fields: namely, the existence of uniformly definable 
G\"odel functions proved in that article implies that 
each global field $K$ is axiomatizable by a single sentence 
$\theta^{\rm Ru}_K$ in the class of global fields, i.e. 
$\theta^{\rm Ru}_K$ holds in a global field $L$ if and only 
if $L \cong K$. This was extended and sharpened by 
the second author in \cite{P_EEI2}, by showing that 
for every finitely generated field $K$ of {\it Kronecker 
dimension\/} $\dim(K)\leqslant2$ there exists a sentence 
$\theta_K$ such that $\theta_K$ holds in a finitely 
generated field $L$ if and only if $L\cong K$ as fields. 
Here, for arbitrary fields $F$, the Kronecker dimension is
$\dim(F):=\td(F)+1$ if $\chr(F)=0$, respectively 
$\dim(F):=\td(F)$ if $\chr(F)>0$, where $\td(F)$ is 
the {\it absolute\/} transcendence degree of $F$.

In this note we establish the analogue of this stronger 
property for all finitely generated fields $K$, thus in 
particular completely resolving the EEIP; in characteristic 
two, though, our proof is conditional, requiring a version 
of resolution of singularities in algebraic geometry, 
%
%
called above $\mathbb{F}_2$. (See Section~\ref{sec:preliminaries} 
for the version of resolution that we need.)
\begin{theorem}
\label{thm1}
Let $K$ be a finitely generated field. If $\chr(K) = 2$ and 
$\dim(K)>3$, assume that resolution of singularities above 
$\lvF_2$ holds. Then there exists a sentence $\theta_K$ 
in the language of rings such that any finitely generated 
field $L$ satisfies $\theta_K$ if and only if $L \cong K$.
\end{theorem}

Our approach follows an idea of \nmnm{Scanlon} in \cite{Sc}, 
and thereby establishes an even stronger statement, 
giving information about the class of definable sets 
in finitely generated fields. Specifically, it shows that 
the class of definable sets is as rich as possible. One 
way of making this precise (cf.~\cite[Lemma~2.17]{AKNS}) 
is the following statement. (See \cite[Section~2]{Sc} or 
\cite[Section~2]{AKNS} for a discussion of the notion of 
bi-interpretability.)
\begin{theorem}
Let $K$ be an infinite finitely generated field. If $\chr(K) = 2$
and $\dim(K)>3$, assume that resolution of singularities 
above $\mathbb{F}_2$ holds. Then $K$ is bi-interpretable 
with $\Z$ (where both $K$ and $\Z$ are considered as 
structures in the language of rings).
\end{theorem}
Note that while this completely characterizes the 
definable sets in $K$, certain questions of uniformity 
across the class of finitely generated fields are left 
open, see e.g. \cite[Question~1.8]{PoonenUniform}.

The chief technical result on which the theorems 
above build, and indeed the result that occupies the 
bulk of this article, concerns a definability statement 
regarding \emph{prime divisors} of finitely generated 
fields. Recall that a \emph{prime divisor} of an arbitrary
field $K$ with $\dim(K)$ finite is any discrete valuation 
$v$ whose residue field $Kv$ has $\dim(Kv)=\dim(K)-1$. 
For finitely generated fields $K$, a valuation $v$ is a 
prime divisor of $K$ if and only if $\dim(Kv)=\dim(K)-1$,
 see e.g.\ \cite[Theorem 3.4.3]{EnglerPrestel}.
A prime divisor $v$ is called \emph{geometric} if 
$\chr(K) = \chr(Kv)$ and \emph{arithmetic} otherwise. 
Throughout, we freely identify valuations 
$v$ with their valuation rings $\clO_v$, 
and in particular do not distinguish between equivalent valuations.

Since the cases $\dim(K)\leqslant 2$ were treated already in \cite{P_EEI2} and \cite{Rumely}, we will consider the following family of hypotheses indexed by $d \geqslant 3$:
\begin{displaymath}({\rm H}_d)\quad
      \quad\left\{ 
      \begin{array}{ll}
             \text{-} \,\,K \text{ is finitely generated with $\dim(K)=d$}.\\
             \text{- If } \chr(K) = 2 \hbox{ \rm and } d >3, 
             \hbox{ \rm resolution of singularities holds above } \mathbb{F}_2.
      \end{array}\right.
\end{displaymath}

\begin{theorem}
\label{thm2}
%
%
Let $d \geqslant 3$.
The geometric prime divisors of fields satisfying $(\HH_d)$ are uniformly first-order
  definable. In other words, there exists a formula $\val_d(X, \underline Y)$ in the language of rings such that for every field $K$ satisfying $(\HH_d)$ and every geometric prime divisor $\mathcal{O}$ of $K$ there exists a tuple $\underline y$ in $K$ such that
  \[ \mathcal{O} = \{ x \in K \colon K \models \val_d(x, \underline y) \}, \]
  and conversely, for every tuple $\underline y$, the subset of $K$ defined above is either a geometric prime divisor or empty.
\end{theorem}

\subsection{Short historical note and the genesis of this article}

The first step in the resolution of the strong form of the EEIP 
as mentioned in Theorem~\ref{thm1} above is \nmnm{Rumely}'s 
work~\cite{Rumely}, which itself builds on previous ideas of 
\nmnm{J.\ha2Robinson}. The next major step toward the 
resolution of the strong EEIP was the introduction of the 
``Pfister form machinery'' in \cite{P_EEI}, followed by the work of
\nmnm{Poonen}~\cite{PoonenUniform}, providing (among other 
things) uniform first-order formulas to define the maximal global 
subfields of finitely generated fields, and \nmnm{Scanlon}~\cite{Sc},
which reduces the strong EEIP to first-order defining the geometric prime 
divisors of finitely generated fields, and finally the introduction 
of the cohomological higher local-global principles (LGPs) in \cite{P_EEI2}, 
as a tool for recovering prime divisors. 
The present paper is a synthesis of previous separate approaches 
to the problem by the authors and supersedes the manuscripts 
\cite{Pop_Distinguishing1, Pop_Distinguishing2, D_DefiningSubrings1, D_DefiningSubrings2}, which are not intended for publication anymore.
The proof builds on and expands 
the above ideas and tools, but it is not  
a straightforward extension of the methods of
\cite{Rumely},~\cite{P_EEI2}, especially
because the higher LGPs involved, cf.~\cite{K-S},~\cite{Ja},
lead to some additional complications compared to the Brauer--Hasse--Noether LGP for global 
fields, respectively Kato's LGP in the case Kronecker dimension two.
Finally, in this note the authors do not discuss the natural
question of the complexity of the formulas describing
prime divisors, thus the sentences characterizing the isomorphism 
type.
%
%
It would also be interesting to treat the EEIP for fields which are finitely generated
over natural base fields such as $\mathbb{C}$, $\mathbb{R}$ and $\mathbb{Q}_p$, cf.~\cite{PoonenPop}.
%
%

\subsection{Acknowledgements}

We would like to thank the anonymous referees for many helpful remarks and suggestions. A major part of this work was completed while the first author was a postdoctoral fellow, and the second author an Eisenbud research professor, in the Definability, Decidability \& Computability (DDC) Program at the MSRI Berkeley during the Fall 2020 semester. The DDC Program was supported by the US National Science Foundation under Grant DMS-1928930. The second author was also partially supported by the NSF FRG Grant DMS-2152304.

\section{Preliminaries: Cohomological Local--Global Principles (LGP)}
\label{sec:preliminaries}

The proof for the definability of prime divisors is
based on local--global principles for certain cohomology 
groups over fields which were introduced in \cite{Kato_LGP}. 
These extend the well-known Brauer--Hasse--Noether 
LGP, in particular the injectivity of the canonical map
\[ 
\imath_K:\Br(K) \hor{} \bigoplus_v \Br(K_{\hat v}),
\]
where $K$ is a global field, the sum is over all places 
$v$ of $K$, and $K_{\hat v}$ is the completion at $v$.
\vskip2pt 
Recall that for an arbitrary field $K$ and $i \in \Z$, one 
defines the $G_K$-modules $\Z/n(i)$ as follows: 
First, if $\chr(K)$ does not divide $n$, then
$\Z/n(i):=\mu_n^{\otimes i}$ is $\Z/n$ endowed with 
the $G_K$-action via the $i^{\rm th}$-power of the 
cyclotomic character of $G_K$. Second, if 
$p:=\chr(K)>0$~and $n=mp^r$ with $(m,p)=1$, then 
$\Z/n(i)\!:=\Z/m(i)\oplus W_r\,\Omega^i_{\rm log}[-i]$, 
where $W_r\,\Omega_{\rm log}$ the logarithmic part 
of the de Rham--Witt complex on the \'etale site over 
$K$ (see \nmnm{Illusie}~\cite{Ill}, Ch.~I, 5.7).
(Note that these two definitions agree when $\chr(K)$ is
positive and does not divide $n$.)
With these 
notations, one has, see \cite{Kato_LGP}, Introduction:
\[
\HHx1K0=\Hom_{\rm cont}(G_K,\lvZ/n),\quad 
  \HHx 2 K 1={}_n\Br(K).
\]
Noticing that $K$ is a global field precisely if $\,\dim(K)=1$, 
and the Brauer\ha1--\ha1Hasse\ha1--Noether local-global 
principle is an LGP for $\HHx 2K1$, \nmnm{Kato} 
proposed that for ``arithmetically significant'' fields $K$ with 
$\dim(K)=d$, e.g.\ for finitely generated fields, there should 
hold similar LGPs for $\HHx{d+1} K d$, see  \nmnm{Kato}'s 
seminal paper~\cite{Kato_LGP}, in particular 
for
how Milnor K-theory plays into the bigger picture. In the same paper,
\nmnm{Kato} proved several forms of such LGPs for finitely 
generated fields $K$ with $\dim(K)=2$. There was/is steady 
progress on Kato's conjectures, see \nmnm{Kerz-Saito}
\cite{K-S} and \nmnm{Jannsen}~\cite{Ja}, where both more
literature and an account of previous results can be found.
\vskip2pt
We mention below three special instances of these 
(much more general) results which we will need in 
the sequel. We consider the following context:
\vskip5pt
$\bullet$ \ Throughout the paper $\bm{n=2}$, and to simplify
notation set $\LL=\lvZ/2$.\footnote{\ha2The 
facts in the remainder of this section hold for 
$\LL=\lvZ/\ell^e\!\!$, provided $\ell\neq\chr(K)$ and 
$\mu_{\ell^e}\subset K$.}
\vskip3pt
$\bullet$ \ For arbitrary fields $F$ and $i\geqslant0$, denote 
$\HHy{i+1}F:=\HH^{i+1}\big(F,\LL(i)\big)$.\footnote{\ha2Note
that in \cite{EKM} one 
denotes $\HHy{i+1}F:=\HH^{i+1}\big(F,\LL(i)\big)$ in \S16,
and $\HHy{i}F:=\HH^{i}\big(F,\LL(i)\big)$ in \S101.} 
\vskip5pt
\noindent
For a field $F$, recall the following general facts:
\vskip3pt
a) For any extension $E|F$ one has the {\it restriction\/} 
map $\res_{E|F}:\HHy{i+1}F\to\HHy{i+1}{E}$, 
$\bm\alpha\mapsto\bm\alpha_E$. 

b) Let $w$ be a discrete valuation on $F$ with residue field $Fw$.
Under mild hypotheses, which are always satisfied in the sequel,
there is a boundary homomorphism
\[
\partial_w: \HHy{i+1}F \to \HHy i {\\Fw} 
\]
(see \cite{Kato_LGP}, p.~149).
By construction, it factors through $\HHy{i+1}{F_w}$, where
$F_w$ is the henselization of $F$ with respect to $w$.

The first higher dimensional LGP proposed by 
\nmnm{Kato} in \cite{Kato_LGP} is
\nmnm{Jannsen}~\cite{Ja},~Theorem~0.4. We  
consider and explain it in our notation for $n=2$. Let 
$K$ be finitely generated of Kronecker dimension 
$d\geqslant1$ and $\ko\subset K$ be a global subfield 
which is relatively algebraically closed in $K$. Then 
$K|\ko$ is a finitely generated field over $\ko$
with $\td(K|\ko)=d-1$.
Let $\lvP(\ko)$ denote the set of places of $\ko$ and 
$\hatkov$ be the completion of $k$ at $v\in\lvP(\ko)$. 
Then the relative algebraic closure
$\kov\subset\hatkov$ of $\ko$ in $\hatkov$ 
satisfies: $\kov$ is the real closure of $k$ at $v$ 
if $v$ is a real place, $\kov=\oli\lvQ$ if $v$ is a 
complex place, respectively $\kov$ is the henselization 
of $\ko$ at finite places $v\in\lvP_{\rm fin}(\ko)$. 
Since $\ko$ is relatively algebraically closed in 
$K$ and $\kov$ is separable over $\ko$,
$K\otimes_{\ko}\kov$ is a domain, 
hence $\hatvK\!:=K\hatkov:={\rm Quot}(K\otimes_{\ko}\kov)$ 
is a well-defined field. In this notation,  \nmnm{Jannsen}
\cite{Ja}, Theorem~0.4, $n=2$ and ${\rm char}(K)\neq 2$ shows that the canonical map 
$\imath_{\ko}\!=\oplus_{v\in\lvP(\ko)}\res_{\hatv K | K}:
\HHy{d+1}K\to\bigoplus_{v\in\lvP(\ko)}\HHy{d+1}{\hatvK}$
is well-defined and injective. 
(Note that \nmnm{Jannsen} writes $F$ for our $K$,
$K$ for our $\ko$, and $F_v$ for our $\hatvK$.)
%
%
Hence if $K_v=K\kov\subset \hatvK$ is the compositum 
of $\kov$ and $K$ inside $\hatvK$, 
setting $\bm\alpha_v\!:=\res_{K_v|K}(\bm\alpha)$,
one gets the following.
\begin{fact}[{cf.\ha2\nmnm{Jannsen}~\cite{Ja},~Thm~0.4, for $n=2$}]
\label{Jannsen1}
\ {\it Suppose that $\chr(K)\neq2$. Then one has:
\vskip2pt
\centerline{$\bm\alpha\in\HHy{d+1}K$ equals $\,0$ if and only if  
$\bm\alpha_v\in \HHy{d+1}{K_v}$ equals $\,0$ for all $v\in\lvP(\ko)$.}}
\end{fact}
%
%
We next briefly recall the higher dimensional generalizations 
of the Brauer\ha1--\ha1Hasse\ha1--\ha1Noether LGP 
as proposed by \nmnm{Kato}. These involve so-called 
{\it arithmetical Bloch--Ogus complexes,\/} see 
\nmnm{Kato}~\cite{Kato_LGP},~\S1 for details.
Namely, for an excellent normal integral scheme $X$
with $\dim(X)=d$ and function field $K=\kappa(X)$, 
let $X_i=X^{d-i}$ be the set of points $x\in X$ 
with $\dim(x)\!:=\dim\oli{\{x\}}=i$, or equivalently,
$\codim(x)\!=d-i$. Under mild 
hypotheses on $X$, which are always satisfied in 
the situations we consider, \nmnm{Kato} shows
(see \cite{Kato_LGP},~Proposition 1.7)
that one has a complex (with the first term
placed in degree $d$):
\footnote{Actually, this is a
special case of the more general context in 
\cite{Kato_LGP}.}
\[
\HHy{d+1}{K}\hor{\partial_d}
{\textstyle\bigoplus}_{x\in X_{d-1}}\HHy{d}{\kappa(x)}
\to\dots\to 
{\textstyle\bigoplus}_{x\in X_{1}}\HHy{2}{\kappa(x)}\to
{\textstyle\bigoplus}_{x\in X_{0}}\HHy{1}{\kappa(x)}.
\leqno{\indent C^0_n(X):}
\]  
%
\vskip-2pt
The first map $\partial_d$ 
%
%
is defined in terms of the discrete 
valuations of $K$ defined by the points $x$ in $X_{d-1}=X^1$
as follows: Since $X$ is normal, the local ring $\clO_x$ 
is a DVR, say, with canonical valuation $w_x$ and residue 
field $\Kw_x=\kappa(x)$. Hence every $x\in X^1$ gives rise to a residue map 
$\partial_x: \HHy{i+1}{K}\to\HHy{d}{\kappa(x)}$
as indicated at~b) above, and one has
%
%
$\partial_d\!:=\oplus_{x\in X^1}\partial_x$. 
\vskip2pt
Let $K_{w_x}$ be the henselization
of $K$ at $w_x$, so $K_{w_x}w_x = Kw_x = \kappa(x)$.
For $\bm\alpha\in\HHy{d+1}K$, recall
its image $\bm\alpha_{w_x}\in\HHy{d+1}{K_{w_x}}$ as
defined at item~a) above. Then by definition, one has 
$\partial_x(\bm\alpha)=\partial_{w_x}(\bm\alpha_{w_x})$ in
$\HHy d{\kappa(x)}$.
Hence if $\HH_d\big(C^0_n(X)\big)=0$ (i.e., if the first
map of the complex is injective), one has:
\vskip5pt
\centerline{\hbox{\large$(*)\ha{40}$} 
$\bm\alpha\in \HHy{d+1}K$ is trivial \ iff \
$\bm\alpha_{w_x}\in \HHy{d+1}{K_{w_x}}$ is trivial for all $x\in X^1.\ha{30}$}
\vskip5pt
Among several other things, \nmnm{Kato} proves in~\cite{Kato_LGP}, 
Corollary, p.~145, that $\HH_2\big(C^0_n(X)\big)=0$ for 
a two-dimensional projective regular integral $\Z$-scheme $X$
such that $K=\kappa(X)$ has no orderings. 

The generalization of \nmnm{Kato}'s result above 
to higher dimensions suitable for our purposes is
given by
(some special form of more general) results by 
\nmnm{Jannsen}~\cite{Ja} and \nmnm{Kerz--Saito}~\cite{K-S},
see~Fact~\ref{KerzSaito} and~Fact~\ref{Jannsen2} below. 
\vskip3pt
Let $R$ be either a finite field with $\chr\neq2$, 
or the valuation ring of a Henselization of a global field~$k$ 
at some $v\in\lvP_{\rm fin}(k)$ such that $\chr(kv)\neq2$. 
Let $X$ be a proper regular integral flat $R$-scheme, 
$K=\kappa(X)$ be its field of rational functions, 
$d=\dim X=\dim K > 0$, and notice that $X$ is excellent and 
$n=2$ is invertible on $X$. \nmnm{Kerz--Saito}~\cite{K-S} 
denote the Kato complex $C^0_n(X)$ introduced above
by ${\rm KC}(X,{\bf Z}/n{\bf Z})$ and its homology by 
${\rm KH}_a(X,{\bf Z}/n{\bf Z})$. This being said, 
Theorem~8.1 of loc.cit.\ for $a=d$ and $\Lambda=\Z/2$ 
asserts that ${\rm KH}_a(X,\Lambda)\big)=0$, that is, 
$\HH_d\big(C^0_2(X)\big)=0$ in the notation of \nmnm{Kato}.
Hence by~$(*)$ above one has the following.  
\begin{fact}[cf.\ \nmnm{Kerz--Saito}~\cite{K-S}, 
Theorem~8.1, for $a=d$, $l=2$, $\Lambda=\lvZ/2\lvZ$]
\label{KerzSaito}
{\it Let $R$, $X$ and $K=\kappa(X)$ be as above. 
Then for $\bm\alpha\in\HHy{d+1}K$ one has:
\vskip5pt
\centerline{$\bm\alpha\in\HHy{d+1}K$ is trivial \ iff \
$\bm\alpha_{w_x}\in \HHy{d+1}{K_{w_x}}$ is trivial 
for all $x\in X^1$.}}
\end{fact}

Finally, we consider the case $\chr=2=n$.
Following \nmnm{Jannsen}, see \cite[Definition 4.18]{Ja}, 
we say that 
\emph{resolution of singularities holds above $\mathbb{F}_2$} if 
the following hold: 
\vskip2pt
\itm{32}{
\item[(i)$\,$] For any proper integral $\lvF_2$-variety $X$, there 
is a proper birational morphism $\tilde X \to X$, where $\tilde X$ 
is a smooth (or equivalently regular) $\lvF_2$-variety. 
\vskip2pt
\item[(ii)] Every affine smooth $\lvF_2$-variety $U$ has an
open immersion $U \hookrightarrow X$, where $X$ is a 
projective smooth $\lvF_2$-variety, and $X\backslash U$ 
is a simple normal crossings divisor.
}

Resolution of singularities is well known for surfaces
and holds in dimension three (in general) by
\nmnm{Cossart--Piltant}~\cite{C-P}. Further, if 
resolution of singularities above $\mathbb{F}_2$ 
holds, then any finitely generated field of characteristic two 
has a smooth proper model over $\F_2$.
\vskip5pt
This being said, Fact~\ref{Jannsen2} below 
follows from results by several authors, e.g.\ 
\nmnm{Kato}~\cite{Kato_LGP} for $\dim(K)=2$, 
\nmnm{Suwa}~\cite[p.~270]{Suwa} for $\dim(K)=3$, 
and (conditionally) \nmnm{Jannsen}
\cite[Thm~0.10]{Ja} for $\dim(K)$ arbitrary. Namely,
let $K$ be a finitely generated field with $\chr(K) = 2$ and 
if $d=\dim(K)>3$, suppose that resolution of singularities 
holds above $\mathbb{F}_2$. Let $X$ be a projective 
smooth $\lvF_2$-model for $K$. Then noticing that
\nmnm{Jannsen}~\cite{Ja} denotes Kato's complex 
$C^0_n(X)$ introduced above by $C^{1,0}(X,\Z/n\Z)$,
by \nmnm{Jannsen}~\cite{Ja}, Thm~0.10, for $a=d$  
and $n=2$, one has that $\HH_a\big(C^{1,0}(X,\Z/n\Z)\big)=0$,
that is, $\HH_d\big(C^0_2(X)\big)=0$ in the notation 
of \nmnm{Kato}. Hence by the discussion at~$(*)$ 
above one has the following. 
\begin{fact}[cf.\ \nmnm{Jannsen}~\cite{Ja}, Thm~0.10, 
for $a=d$ and $n=2$]
\label{Jannsen2}
{\it In the above notation and hypothesis, for all
$\bm\alpha\in\HHy{d+1}K$ the following holds:
\vskip5pt
\centerline{$\bm\alpha\in\HHy{d+1}K$ is trivial \ iff \
$\bm\alpha_{w_x}\in \HHy{d+1}{K_{w_x}}$ is trivial 
for all $x\in X^1$.}}
\end{fact}

\section{Consequences/Applications of the 
                        Local-Global Principles}
\label{sec:Florian1}
\noindent
We begin by recalling a few basic facts about 
{\it Pfister forms,\/} which are at the core of first-order 
definability of prime divisors. 
%
%
For an field $F$ and $a\in F^\times$ set $\Pff a=x_1^2-ax_2^2$,\footnote{\ha2Some other 
sources prefer the convention $\Pff a=x_1^2 + ax_2^2$ in the case ${\rm char}(F)\neq2$.}
respectively $\Pft a:=x_1^2+x_1x_2+ax_2^2$. For an $(i+1)$-tuple $\bma=(a_i,\dots,a_0)$ 
with $a_i,\dots,a_0\in F^\times\!$, the $(i+1)$-fold Pfister form $q_{\bm a}$ is defined as follows, see \cite[9.B]{EKM} for details:
\vskip4pt
\itm{15}{
\item[-] If $\chr(F)\neq2$, then 
%
%
$q_{\bm a}\!:=q_{a_i,\dots,a_0}\!:=\Pff{a_i}\otimes\dots\otimes\Pff{a_0}$.
\vskip2pt
\item[-] If $\chr(F)=2$, then 
%
%
$q_{\bm a}\!:=q_{a_i,\dots,a_0}\!:=
\Pff{a_i}\otimes\dots\Pff{a_1}\otimes\Pft{a_0}$.\footnote{\ha2In this case, one could allow $a_0 = 0$ without harm, but we prefer to require all $a_i\neq0$ for uniformity.}
}
It is well known, see \cite[Corollary 9.10]{EKM}, that 
a form $q_{\bm a}$ as defined above is 
isotropic if and only if it is hyperbolic. 
Further, recalling that $\HHy{i+1} F:=\HH^{i+1}\big(F,\LL(i)\big)$
with $\LL=\lvZ/2$ as introduced above, 
by \cite{EKM}, Section 16,\footnote{\ha2Be aware 
of the inconsistency of notation in \cite{EKM}, see 
footnote 2 of this article.}  to every Pfister form $q_{\bm a}=\Pff{\bma}$ or 
$q_{\bm a}=\Pft{\bma}$, one can attach in a canonical way a 
cohomological invariant
\[ 
e(q_{\bm a})\in\HHy{i+1}F.
\]
Let $N\!:=2^{i+1}-1$. Then $q_{\bm a}$ is a quadratic 
form in $N+1$ variables $\bm x=(x_1,\dots,x_{N+1})$, and 
\vskip5pt
\centerline{\it the associated variety $V_{q_{\bm a}}\!:=V_F(q_{\bm a})\hra\lvP^N_F\,$
is a smooth $F$-subvariety of $\lvP^N_F$.\/}
\begin{fact}
\label{fkt0}
{\it 
%
%
In the above notation, the following hold:
\vskip2pt
\begin{itemize}[leftmargin=25pt]
\item[{\rm1)}] The Pfister form $q_{\bm a}$ is isotropic over 
$F$ if and only if $e(q_{\bm a})=0$ in $\HHy{i+1}F$.
\vskip2pt
\item[{\rm2)}] Let $E|F$ be a field extension, and 
$q_{\bm a,E}$ be $q_{\bm a}$ viewed over~$E$. 
One has $e(q_{\bm a,E})=\res\big(e(q_{\bm a})\big)$ 
under $\res_{E|F}:\HHy{i+1}F\to\HHy{i+1}E$.
\end{itemize}
}
\end{fact}
\noindent
Concerning the proofs, assertion~1) is implied by the 
Milnor Conjecture (although 
%
%
previous weaker results would suffice, see
\cite{ElmanLam_MilnorConj, Kato_Quadratic}; 
to be precise, use \cite[Fact 16.2]{EKM} together with
the fact that the $(i+1)$-fold Pfister form $q_{\bm a}$ is
isotropic if and only if it is hyperbolic,
which is the case if and only if
its class in the Witt ring of $F$ lies in $I_q^{i+2}(F)$
\cite[Theorem 23.7(1)]{EKM}).
Assertion~2) 
follows by definition.
\vskip2pt
We conclude this preparation with the following facts scattered 
throughout the literature (although some of them might be new 
in the generality presented here); variants of these will be used 
later. For the reader's sake we give the (straightforward) full proofs.  

\begin{proposition}
  \label{prop1}
  \label{prop:basicProperties}
Let $F$ be henselian with respect to a non-trivial non-dyadic 
valuation $w$, i.e., 
$\big(\chr(F), \chr(\Fw)\big) \neq (0, 2)$.
Let $k\subset F$ be its constant subfield, i.e., the relative 
algebraic closure of the prime subfield in $F$.
Let $\bm\veps=(\veps_r,\dots,\veps_0)$ be $w$-units in $F$. 
\vskip2pt
\begin{itemize}[leftmargin=25pt]
\item[{\rm1)}] Suppose that $w(\veps_1-1)\!>\!0$. Then 
$q_{\veps_1\hb1,\veps_0}$ is isotropic over~$F$. Hence
$q_{\bm\veps}$ is isotropic over~$F$.
\vskip2pt
\item[{\rm2)}] Let $\bm{\oli\veps}$ be the image of $\bm\veps$ under the 
residue map $\clO_w^\times\to \Fw$, and $\bm\pi=(\pi_s,\dots,\pi_1)$,
$\pi_i\in F^\times$ be such that $w(\pi_s),\dots,w(\pi_1)$ are 
${\lvF}_2$-independent in $wF/2$. The following are equivalent:
\vskip2pt
\centerline{\ha{20}{\rm(i)} $q_{\overline{\bm\veps}}$ is isotropic 
over $\Fw\,;$ \ {\rm(ii)} $q_{\bm\veps}$ is isotropic over $F\,;$ \
{\rm(iii)} $q_{(\bm\pi,\bm\veps)}$ is isotropic over~$F$.}
\vskip2pt
\item[{\rm3)}] Suppose that $\,\dim(F)= r$, and $q_{\bm\veps}$ 
is isotropic over the compositum $F_v\!=k_v F$ for each
real closure $k_v$ of $k$ $($if there are any such $k_v)$.
Then $q_{\bm\veps}$ is isotropic over $F$.
\end{itemize}
\end{proposition}

\begin{proof} Let $N\!:=2^{r+1}\!-1$, 
%
%
and recall the $(r+1)$-fold Pfister form 
$q_{\bm\veps}=q_{\bm\veps}(\bm x)$ in variables 
$\bm x=(x_1,\dots,x_{N+1})$. Since $w$ is non-dyadic, 
$V_{q_{\bm\veps}}\hra\lvP^N_{\clO_w}$ is a smooth 
$\clO_w$-subvariety of $\lvP^N_{\clO_w}$, with 
special fiber the projective smooth $Kw$-variety 
$V_{q_{\oli{\bm\veps}}}\hra\lvP^N_{\Fw}$. Hence by 
Hensel's Lemma, the  specialization map on rational points 
$V_{q_{\bm\veps}}(F) \to V_{q_{\oli{\bm\veps}}}(\Fw)$
is surjective, implying:
\vskip4pt
$(*)$ \ {\it $q_{\bm\veps}$ is isotropic over $F$ if 
 and only if $q_{\oli{\bm\veps}}$ is isotropic over $\Fw$.}
%
\vskip3pt\noindent
To 1): Since $\oli\veps_1=1$, $q_{\oli\veps_1,\oli\veps_0}=q_{1,\oli\veps_0}$ is 
isotropic over $\Fw$, thus so are $q_{\veps_1,\veps_0}$
and $q_{\bm\veps}$ over $F$~by~$(*)$.
\vskip2pt
\noindent
To 2): Setting $\bm x_\chm=(x_{\chm, i})_{i\leqslant N}$, 
$\pi_\chm=\prod_i\pi_i^{\chm(i)}$ for $\chi:\{1,\dots,s\}\to\{0,\!1\}$, 
$\bm y=(\bm x_\chm)_\chm$, one has $q_{\bm\pi,\bm\veps}(\bm y)
=\sum_\chm \pi_\chm q_{\bm\veps}(\bm x_\chm)$.
By $(*)$ above, $q_{\oli{\bm\veps}}$ is isotropic over 
$\Fw$ if and only if $q_{\bm\veps}$ is isotropic over $F$, and if so, 
$q_{\bm\pi,\bm\veps}$ is isotropic over $F$. For the 
converse, let $q_{\oli{\bm\veps}}$ be anisotropic. Then 
$w\big(q_{\bm\veps}(\bm\nu)\big)\in2\cdot wF$ for all 
$\bm\nu\neq0$ in $F^N\!$, and for $\bm\mu=
(\bm\nu_\chm)_\chm\neq\bm0$, one has:  Since 
$w(\pi_\chm)=\sum_i\chi(i)w(\pi_i)$, and 
$w\big(q_{\bm\veps}(\bm\nu_\chm)\big)\in2\cdot wF$, and 
$\big(w(\pi_i)\big)_i$ are independent in $wF/2$, it follows 
that the summands in $q_{\bm\pi,\bm\veps}(\bm\mu)=
\sum_\chm\pi_\chm q_{\bm\veps}(\bm\nu_\chm)$ have 
distinct values. Hence $q_{\bm\pi,\bm\veps}(\bm\mu)
\neq0$, thus $q_{\bm\pi,\bm\veps}$ is anisotropic.
\vskip2pt
To 3): We first claim that
$e:=\dim(\Fw)<\dim(F)=r$. Indeed, by the Abhyankar 
Inequality, see e.g. \cite{EnglerPrestel}, Thm~3.4.3, 
one has: $\td(F)-\td(\Fw)\geqslant r(w)$, where 
$\td(\bullet)$ is the absolute transcendence degree 
and $r(w)\!:= \dim_\Q((wF/wk)\otimes\Q)$ is the rational rank
of the abelian group $wF/wk$. First, if $w|_k$ is
non-trivial, then $\chr(k)=0$ and $kw$ is algebraic 
over a finite field and therefore
$\dim(F)-\dim(\Fw)=1+\td(F)-\td(\Fw)
\geqslant 1+r(w)>0$. Second, if $w|_k$ is trivial, then $r(w)>0$, 
hence $\dim(F)-\dim(\Fw)=\td(F)-\td(\Fw)\geqslant r(w)>0$. 
\vskip2pt
\underbar{Case 1}. \ $\chr(\Fw)=p>0$. Then $e=\dim(\Fw)=\td(\Fw)$, 
and $q_{\oli{\bm\veps}}$ is a quadratic form in $2^{r+1}$ variables 
over $\Fw$. Since $e<r$, and $\Fw$ is a $C_{e+1}$-field, 
$q_{\oli{\bm\veps}}(\bmx)=0$ has non-trivial solutions in $\Fw$, 
i.e., $q_{\oli{\bm\veps}}$ is isotropic over $\Fw$. Hence 
so is $q_{\bm\veps}$ over $F$ by~$(*)$. 
\vskip2pt
\underbar{Case 2}. \ $\chr(\Fw)=0$. \ Then $w$ is trivial on
the constant field $k$ of $F$, and by Hensel's Lemma, 
there is a field of representatives $E\subset F$ for $\Fw$. 
Further, $E$ is relatively algebraically closed in $F$,
so $k\subset E$ is relatively closed in $E$, and  
$\dim(E)=\dim(\Fw)=e<r$. Let ${\bm\eta}=(\eta_r,\dots,\eta_0)\in E^{r+1}$ 
be the lifting of $\oli {\bm\veps}=(\oli \veps_r,\dots,\oli \veps_0)$. 
Then $\veps_i=\eta_i \delta_i$ with $\delta_i \in F$ 
and $w(\delta_i-1)>0$. Since ${\rm char}(Fw)=0$,
by Hensel's Lemma, each $\delta_i$ is a square in $F$, thus 
$q_{\bm\veps}\approx q_{\bm\eta}$ over $F$, and 
$q_{\bm\eta}$ is defined over $E\subset F$. We consider
the following condition on subfields $E'\subset E$:
\vskip5pt
\centerline{$(*)\ha{80}$ \it $E' $ is finitely generated,
$\eta_r,\dots,\eta_0\in E'$.$\ha{80}$\/}
\vskip5pt
For every $E'$ satisfying $(*)$, consider $k'\!:=k\cap E'$ 
and for $v'\in\lvP(k')$, let $E'_{v'}=k'_{v'}E'$ be the 
compositum of $E'$ and $k'_{v'}$ and $e(q_{\bm\eta,v'})$
be the cohomological invariant of $q_{\bm\eta}$ in
$\HH^{r+1}(E'_{v'})$. 
\vskip5pt 
{\bf Claim.} {\it There is $E'\subset E$ satisfying $(*)$  
such that $e(q_{\bm\eta,v'})=0$ for all $v'\in\lvP(k')$.\/}
\vskip5pt
{\it Proof of the Claim.}  Let $E'$ satisfy $(*)$,
$k'=E'\cap k$. First, if $v'\in\lvP(k')$ is 
not a real place, then $E'_{v'}$ has no orderings; hence
by the well-known behaviour of cohomological dimension
in field extensions we have
$\cd(E'_{v'})\leqslant\cd(E')\leqslant\dim(E')+1\leqslant\dim(Fw)+1<r+1$. Thus 
$\HH^{r+1}(E'_{v'})=0$, implying that $e(q_{\bm\eta,v'})=0$.
Second, concerning real places of $k'$, let 
$\Sigma_{E'}\subset\lvP(k')$ be the (possibly 
empty) set of all real places $v'$ such that  
$e(q_{\bm\eta,v'})\neq0$. 
By contradiction, suppose that $\Sigma_{E'}$ is 
non-empty for all $E'$ satisfying~$(*)$. Considering
all $E'\subset E''\subset E$ satisfying~$(*)$, the 
restriction maps $\Sigma_{E''}\to\Sigma_{E'}$ make
$(\Sigma_{E'})_{E'}$ into a projective system of 
finite non-empty sets, 
having as projective limit the non-empty set 
$\Sigma_E\subset\lvP(k)$ of all $v\in\lvP(k)$ which
satisfy $v'\!:=v|_{k'}\in\Sigma_{E'}$ for all $E'$
(where as always $k' = E' \cap k$). For
$v\in\Sigma_E$, let $k_v$ be the real closure 
of~$k$ at $v$ and $w_v|w$ be the unique prolongation 
of the Henselian valuation $w$ of $F$ to the algebraic 
extension $F_v=k_v F$ of $F$. Since $w$ is trivial 
on $k$, the residue field $F_vw_v$ is the compositum 
$k_v \Fw$, and further, $E_v\!:=k_v E\subset F_v$ 
is a field of representatives for $k_v \Fw$. Since 
$q_{\bm\eta}$ is isotropic over $F_v$, it is so over 
$k_v \Fw$, hence over $E_v=k_v E$. Equivalently,
by Fact~3.1,  
$e(q_{\bm\eta,v})=0$ in $\HH^{r+1}(E_{v})$. 
On the other hand, since cohomology is compatible 
with inductive limits, $e(q_{\bm\eta,v})=\lim\hb{16}
\lower5pt\hbox{$\scriptstyle\longrightarrow$}\,
e(q_{\bm\eta,v'})\neq0$, because $e(q_{\bm\eta,v'})\neq0$
for all $E'$, contradiction! The Claim is proved.
\vskip2pt
Back to the proof in Case 2), let $E'\subset E$ satisfy the 
Claim. Set $F'=E'(\bmt)$ for $\bmt$~a~transcendence 
basis of $F|E'$.
Then~$F'\subset F$~is~finitely generated, 
$E'\cap k=k'=F'\cap k$ and $E'_{v'}\subset F'_{v'}\!:=k'_{v'}F'$ 
for all $v'\in\lvP(k')$. Since $e(q_{\bm\eta,v'})=0$ 
in $\HH^{r+1}(E'_{v'})$, $q_{\bm\eta}$ is isotropic over 
$E'_{v'}$, hence over $F'_{v'}$ for each $v'\in\lvP(k')$. 
Hence by Fact~3.1, $e(q_{\bm\eta,v'})=0$ in 
$\HH^{r+1}(F'_{v'})$ for all $v'\in\lvP(k')$, and therefore, by 
Fact~\ref{Jannsen1}, $e(q_{\bm\eta})=0$ in $\HH^{r+1}(F')$.
Equivalently, $q_{\bm\eta}$ is isotropic over $F'\!$,
thus over $F$. Finally, $q_{\bm\veps}\approx q_{\bm\eta}$
is isotropic over $F$.
\end{proof}
\noindent
A) \ {\bf Prime divisors via anisotropic 
$\ko\sps$nice Pfister forms}.
%
%
We now state a technical condition for the Pfister forms we are 
going to work with. This technical condition in particular serves 
to ensure that orderings and dyadic places can always be eliminated 
from our subsequent considerations.
\begin{definition} 
\label{def:nice}
Let $K$ be a field satisfying Hypothesis (H$_d$) from the 
Introduction and $q_{\bm a}$ be a Pfister form defined by 
$\bma:=(a_d,\dts,a_1,a_0)$ with all $a_i\in K^\times$.
\vskip2pt 
\itm{25}{
\item[1)] Let $\ko\subset K$ be a global subfield. We say
that $q_{\bm a}$ is $\ko$-\emph{nice} if $a_1,a_0\in \ko$, 
and the two-fold Pfister form $q_{a_1,a_0}$ satisfies:
\vskip4pt
\item[] $\hb4(*)$ If $v\in\lvP(\ko)$ is real, or dyadic, or 
$v(a_0)\neq0$, or $v(a_1)<0$, then $q_{a_1,a_0}$ is 
isotropic~over~$\kov$.
\vskip2pt
\item[2)] We say that $q_{\bm a}$ is \emph{nice} if there 
is there is a global subfield $\ko\subset K$ such that 
$q_{\bm a}$ is $\ko\sps$nice.
}
\end{definition}
Note that being nice is not an isometry invariant of Pfister 
forms, so strictly speaking it is a property of the concrete 
presentation; this should not lead to confusion.
\vskip2pt
Due to the results of the previous section, we now have the 
following local--global principle for isotropy of nice Pfister forms. 
\begin{proposition}
\label{prop:niceLocalGlobal}
Let $K$ satisfy Hypothesis~{\rm(H$_d$)}, $\ko\subset K$ 
be a global subfield, and $q_{\bm a}$ be an anisotropic
$\ko\sps$nice Pfister form over $K$. The following hold:
\vskip2pt
\itm{25}{
\item[{\rm1)}] There is a prime divisor $w$ of $K$ such
that $q_{\bm a}$ is anisotropic over the 
$w$-henselization~$\hns Kw$.
\vskip2pt
\item[{\rm2)}] If $w$ is a prime divisor of $K$ 
such that $q_{\bm a}$ is anisotropic over the $w$-henselization 
$\hns Kw$, then $w$ is non-dyadic, $w(a_0)=0$, $w(a_1)\geqslant0$, and 
$w(a_i)$ is odd for some $i=1,\dots,d$.
}
\end{proposition}
\begin{proof} To 1): 
By Fact~\ref{fkt0},~1),~2) above, proving that $q_{\bm a}$
is anisotropic over $\hns Kw$ is equivalent to proving that 
the image of $e(q_{\bm a})$ under the restriction map 
$\res_w:\HHy{d+1}K\to\HHy{d+1}{\hns Kw}$ does not vanish. Noticing that 
$e(q_{\bm a})\neq0$ in $\HHy{d+1}K$, proceed as follows:
\vskip2pt
\underbar{Case 1)}. If $\chr(K)=2$, then choosing a smooth 
projective $\lvF_2$-model $X$ for $K$, by~Fact~\ref{Jannsen2} 
above, there is a prime divisor $w$ of $K$, say $w=w_x$ for 
some point $x\in X^1\!$, such that 
$\res_w\big(e(q_{\bm a})\big)\neq0$ in $\HHy{d+1}{\hns Kw}$,
and therefore $q_{\bm a}$ is anisotropic over $\hns Kw$.
\vskip2pt
\underbar{Case 2)}. If $\chr(K)\neq 2$, we apply  
Fact~\ref{Jannsen1} above, so there is $v\in\lvP(\ko)$ 
such that $\res_v(e(q_{\bm a}))\neq0$ in $\HHy{d+1}{K_v}$. Hence if 
$q_{\bm a,v}$ is the Pfister form $q_{\bm a}$ viewed over 
$K_v$, then $e(q_{\bm a,v}) = \res_v\big(e(q_{\bm a})\big) \neq 0$.
Equivalently, $q_{\bm a,v}$ is anisotropic over $K_v$, 
hence its Pfister subform $q_{a_1,a_0}$ is anisotropic over 
$\kovh\subset K_v$. Thus by condition~$(*)$ of Definition~\ref{def:nice} 
above, $v$ is a finite non-dyadic place of $\ko$. In particular, 
letting $R\subset\kovh$ be the henselization of $\clO_v$, 
it follows that $\chr(kv)\neq2$. Let $X_v$ be any 
projective $R$-model of $K_v$. Then using {\it prime 
to $\ell$-alterations with $\ell=2$,\/} see \cite{ILO}, 
Expos\'e~X, Thm~2.4, there is a projective regular irreducible
$R$-scheme $\tlX$ and a projective surjective $R$-morphism 
$\tlX\to X_v$ defining a finite field extension $\tlK\,|\,K_v$ 
of degree prime to~$2$. In particular, the restriction of 
$e(q_{\bm a,v})$ in $\HHy{d+1}{\tlK}$ is non-zero.
Hence by Fact~\ref{KerzSaito} above, there exists 
$\tlx\in \tlX{}^1$ such that setting $\tlw:=w_{\tlx}$, 
for the $\tlw$-henselization $\hns\tlK\tlw$ of $\tlK$
one has: $\res_{\tlw}\big(e(q_{\bm a,v})\big)\neq0$ in 
$\HHy{d+1}{\hns\tlK\tlw}$. Hence letting $q_{\bm a,\tlw}$
be the Pfister form $q_{\bm a}$ viewed over 
$\hns\tlK\tlw$, one has $e(q_{\bm a, \tlw})=
\res_{\tlw}\big(e(q_{\bm a,v})\big)\neq0$ in 
$\HHy{d+1}{\hns\tlK\tlw}$, concluding by Fact~\ref{fkt0}
that $q_{\bm a,\tlw}$ is anisotropic over $\hns\tlK\tlw$. 
Let $w:=\tlw|_K$. Then since $\tlK\,|\,K$ is an
algebraic extension, and $\tlw$ is a prime divisor of 
$\tlK$, it follows that $w=\tlw|_K$ is a prime divisor 
of $K$, and the $w$-henselization $\hns Kw$ is 
contained in $\hns\tlK\tlw$. 
Since $q_{\bm a,\tlw}$ is anisotropic over $\hns\tlK\tlw$, it 
follows that $q_{\bm a}$ is anisotropic over $\hns Kw$.
\vskip2pt
To 2): Let $v:=w|_{\ko}$ be the restriction of $w$ to $\ko$
(which might be the trivial valuation). Then $\kovh$ 
is contained in $\hns Kw$. Hence since $q_{\bm a}$ is 
anisotropic over $\hns Kw$, its subform $q_{a_1,a_0}$ (which
is defined over $\ko$) is anisotropic over $\kovh$. 
Since $q_{a_1,a_0}$ is $\ko\sps$nice, either $v$ is trivial or $v\in\lvP(\ko)$ 
must be finite non-dyadic and $v(a_0)=0$, $v(a_1)\geqslant0$. 
Hence $w$ is non-dyadic, and further,
$w(a_0)=v(a_0)=0$, $w(a_1)=v(a_1) \geqslant 0$.
It remains to show that $w(a_i)$ is odd for some $i = 1, \dotsc, d$.
If not, for all such $i$ we may write $a_i = b_i c_i^2$
for some $b_i, c_i \in K^\times$ with $w(b_i) = 0$.
But then $q_{\bm a} \approx q_{b_d, \dotsc, b_1, a_0}$,
and the latter form is isotropic over $K_w$ by Proposition~\ref{prop1},~3)
(where the hypothesis on real places is satisfied by niceness of $q_{\bm a}$).
Therefore $q_{\bm a}$ is also isotropic over $K_w$ in contradiction to the hypothesis.
\end{proof}
\vskip2pt
\noindent
B) \ {\bf Abundance of anisotropic $\ko\sps$nice Pfister forms}
\vskip2pt

In the subsection~A) above we saw that anisotropic nice 
Pfister forms over a finitely generated field $K$ remain
anisotropic over some henselization
of $K$ w.r.t. some non-dyadic prime 
divisors of $K$. In this subsection, we prove that given any 
geometric prime divisor $w$ of $K$, and a global subfield 
$\ko\subset K$ with $w$ trivial on $\ko$, there are 
``many'' $\ko\sps$nice Pfister forms that remain anisotropic 
over the $w$-henselization $K_w$. Our actual result,
Proposition~\ref{prop:existsTestForm} below, is more 
complicated to state, because we want to realize 
additional restrictions on the Pfister forms.
\begin{lemma}\label{lem:sepExtGlobalAnisoPfister}
Let $l_1/k_1$ be a finite separable extension of global 
fields, and $\Sigma \subset \lvP_{\mathrm{fin}}(k_1)$ 
a finite set of finite places of $k_1$. Then there exists
a $\ko\ha1$-\ha1nice Pfister form $q_{a_1,a_0}$ 
over $\ko$ such
that $v(a_1) = v(a_0) = 0$ for all $v \in \Sigma$ and 
$q_{a_1, a_0}$ is anisotropic over $l_1$.
\end{lemma}
\begin{proof}
We may enlarge $\Sigma$ to contain all dyadic places of 
$k_1$. There are infinitely many finite places of $k_1$
which split completely in $l_1$. Pick one such place 
$v_1$ which is not in $\Sigma$. Using weak approximation, 
choose $a_0 \in k_1^\times$
such that $v(a_0) = 0$ for all $v \in \Sigma$,
and $v_1(a_0) = 0$,
and furthermore the reduction of the polynomial
$X^2 - X - a_0$ in $k_1v_1[X]$ is irreducible if the
characteristic of $k_1 v_1$ is $2$, respectively
the reduction of the polynomial $X^2 - a_0$ in $k_1v_1[X]$
is irreducible if the characteristic of $k_1 v_1$ is not $2$.
(The case distinction here arises from the different definition
of the form $q_{a_0}$ depending on the characteristic.)

Let $l'=\ko(\alpha_0)$, with $\alpha_0$ a root of $X^2 - X - a_0$
respectively $X^2 - a_0$.
Pick a place $v_0 \in \lvP_{\mathrm{fin}}(k_1)\backslash\Sigma$ 
which splits completely in $l'\!$, hence $v_0\neq v_1$ because 
$v_1$ is inert in $l'$. Using the Strong Approximation Theorem, choose 
$a_1 \in k_1^\times$ satisfying the following four conditions: 
$v_1(a_1) = 1$;
$a_1$ is a norm of the local extension $\kov(\alpha_0)|\kov$ for all the 
finitely many $v \in \lvP_{\rm fin}(k_1)$ for which $v(a_0) \neq 0$,
all dyadic $v$ and all real $v$; 
$v(a_1) = 0$ for all $v \in \Sigma$; $v(a_1) \geqslant 0$ at all 
$v \in \lvP_{\rm fin}(k_1)\backslash \{v_0\}$.
(The condition at dyadic $v \in \Sigma$ is thus that $v(a_1) = 0$
and $a_1$ is a local norm, both of which are open conditions
satisfied in a $v$-neighbourhood of $1$.)
The following 
hold: First, $q_{a_1, a_0}$ 
is anisotropic over $k_{1,v_1}$ by the definitions of 
$v_1$, $a_0,a_1$ and Proposition \ref{prop1} 2).
Hence $q_{a_1,a_0}$ is anisotropic 
over $l_1 \subset k_{1,v_1}$. Second, we  claim that $q_{a_1, a_0}$
is $\ko\sps$nice. Indeed, by the choice of $a_1$ one has: 
If $v(a_0)\neq0$ or $v$ is dyadic or $v$ is real, 
then $a_1$ is a norm of $\kov(\alpha_0)/\kov$.
Hence in these 
cases, $q_{a_1,a_0}$ is isotropic over $\kov$. Finally, 
if $v(a_1)<0$, then $v=v_0$, hence 
$v$ is totally split in $l'=\ko(\alpha_0)$, implying that 
$\alpha_0\in\kov$. Hence $q_{a_1,a_0}$ is isotropic~over~$\kov$. 
\end{proof}
\begin{lemma}
\label{lmm2}
Let $K$ satisfy Hypothesis~{\rm(H$_d$)}, and $w$ 
be a geometric prime divisor of $K$. There is a global 
subfield $\ko\subset K$, and $\ko$-algebraically 
independent elements $\bm u=(u_i)_{d>i>1}$ of $K$ 
such that $w$ is trivial on $\ko(\bm u)$ and $\Kw$ is 
finite separable over $\ko(\bm u)$. Moreover, if 
$u_d\in K$ has $w(u_d)=1$, then $(u_d,\bm u)$ is 
a separating transcendence basis of $K|\ko$.
\end{lemma}
\begin{proof}
Since $w$ is geometric, $K$ and $K\!w$ have the same 
prime field $\kappa_0$, and are separably generated 
over $\kappa_0$. Proceed as follows: (i) If $\chr(K)=0$, 
let $(u_i)_{d>i>1}$ be any $w$-units which lift a transcendence 
basis of $\Kw$. (ii) If $\chr(K)>0$, let $(u_i)_{d>i>0}$ be 
$w$-units which lift a separating transcendence basis of $\Kw$. 
Let $\ko\subset K$ be the constant field in case (i), and 
the relative algebraic closure of $\kappa_0(u_1)$ in $K$ in case~(ii),
and set $\bm u=(u_i)_{d>i>1}$ in both cases. Then $w$ is 
trivial on $\ko(\bm u)$, and the residue of $\bm u$ in $\Kw$ is 
a separating transcendence basis of $\Kw$ over $\ko$. 
Assume now that $w(u_d) = 1$; thus in particular, $w$
is not trivial on $\ko(u_d,\bm u)$. Since $w$ is 
trivial on $\ko(\bm u)$ and non-trivial on $\ko(u_d,\bm u)$,
$u_d$ cannot be algebraic over $\ko(\bm u)$. Hence 
since $\td\big(K\,|\,\ko(\bm u)\big)=1$, $(u_d,\bm u)$ is a 
transcendence basis of $K$ over $\ko$, and $K|\ko(u_d,\bm u)$
is a finite field extension.
We claim that $K\,|\,\ko(u_d,\bm u)$ is separable. 
Indeed, let $K_s$ be the separable closure of $\ko(u_d,\bm u)$
in $K$, and set $w_s\!:=w|_{K_s}$. Since $K|K_s$ is purely 
inseparable, $w$ is the only prolongation of $w_s$
to $K$, and further one has: First, $w(u_d)=1=w_s(u_d)$, 
hence $e(w|w_s)=1$. Second, $\Kw\,|\,K_sw_s$ is 
purely inseparable, and since $\Kw\,|\,\ko(\bm u)$ 
is separable~and $\ko(\bm u)\subset K_sw_s$, one 
must have $\Kw=K_sw_s$, hence $f(w|w_s)=1$.
Third, since $K\supset K_s$ are function~fields in one 
variable over $\ko(\bm u)$, the fundamental equality for 
$w_s$ and its unique prolongation $w$ to $K$ holds, 
see e.g.\ \cite{Chevalley_FunctionFields}, Ch.\ha2IV, §1, 
Theorem~1. Hence $[K:K_s]=e(w|w_s)f(w|w_s)=1$, 
and thus $K=K_s$ is separable over $\ko(u_d,\bm u)$.
\end{proof}
\begin{definition}
Let $K$ satisfy Hypothesis (H$_d$), $\ko\subset K$ 
be a global subfield, and $\bm t=(t_i)_{d>i>1}$ be 
$\ko$-algebraically independent in $K$. 
A $\,\ko,\hb1\bmt\sps${\it test form\/} for an element 
$a_d\in K\tms$ is any $\ko\sps$nice Pfister form $q_{\bm a}$ 
defined by $\bma=(a_d,a_{d-1},\dts,a_1,a_0)$, where $(a_i)_{d>i>1}
=\bmt-\bm\epsilon$ and $\bm\epsilon=(\epsilon_i)_{d>i>1}$
are such that $\epsilon_i\in\ko$ for $1\!<\!i\!<\!d$ 
are $v$-units for all $v\in\lvP_{\rm fin}(\ko)$ with $v(a_1)>0$.
\end{definition}
\begin{proposition}
\label{prop:existsTestForm} Let $K$ satisfy Hypothesis
{\rm (H$_d$)} and $w$ be a geometric prime divisor of $K$. 
Let $\ko\subset K$ be a global subfield, $\bmt=(t_i)_{d>i>1}$ 
be $\ko$-algebraically independent elements of $K$ 
such that $w$ is trivial on $\ko(\bmt)$, and 
$K\!w\ha1|\ha1\ko(\bmt)$ is finite separable. Then there is a 
Zariski open dense subset $\,U\subset{\ko^\times}^{d-2}$ 
satisfying: For every $\bm\epsilon=(\epsilon_i)_{d>i>1}\in U\!$,
there is a $\ko\sps$nice Pfister form $q_{a_1,a_0}$, such
that for arbitrary $a_d \in K^\times$ with $w(a_d)$ odd,
setting $(a_i)_{d>i>1}=\bmt-\bm\epsilon$ and 
$\bm a=(a_d,\dots,a_1,a_0)$, one has that $q_{\bm a}$ is
a $\ko,\bmt\sps$test form for $a_d$ which is anisotropic 
over $\hns K w$.
%
\end{proposition}
\begin{proof}
The normalization morphism $S \to S_{\bm t}$ 
of $S_{\bm t}\!:=\Spec\ko[\bm t,\bmt^{-1}]$ in the finite separable field 
extension $l\!:=\Kw\hla\ko(\bmt)$ is a finite generically 
separable cover, thus \'etale above a Zariski open dense subset 
$U_l\subset S_{\bm t}$. Hence for 
$\bm\epsilon:=(\epsilon_i)_{d>i>1} \in U\!:=U_l(\ko)$, any preimage 
$s_{\bm\epsilon}\mapsto\bm\epsilon$ of 
$\bm\epsilon$ under the morphism $S \to S_{\bm t}$
is a smooth point of $S$,
$\bm\pi\hb2:=(a_i)_{d>i>1}:=(t_i-\epsilon_i)_{d>i>1}$ 
is a regular system of parameters at 
$s_{\bm\epsilon}\mapsto\bm\epsilon$,
and the residue field extension $\ko=\kappa(\bm\epsilon)\hra\kappa(s_{\bm\epsilon})=:\lx$ is finite separable. 
In particular, the completion of the local ring $\clO_{s_{\bm\epsilon}}$ is 
the ring of formal power series $\widehat\clO_{s_{\bm\epsilon}}=
\lx\fps{\bm\pi}$ in the variables $\bm\pi=(a_i)_{d>i>1}$ 
over $\lx$. Hence one has $\ko(\bm\pi)$-embeddings
\[
l=\Kw=\Quot(\clO_{s_{\bm\epsilon}})\hra
\Quot(\widehat\clO_{s_{\bm\epsilon}})=
\Quot\big(\lx\fps{a_2,\dts,a_{d-1}}\big)
\hra\lx\lps{a_2}\dts\lps{a_{d-1}}=:\widehat l.
\]
Let $\Sigma\subset\lvP(\ko)$ be any finite set of finite places 
such that all $(\epsilon_i)_{d>i>1}$ are $\Sigma$-units, and 
for $\lo:=\lx$, consider $a_1,a_0\in\ko$ as in 
Lemma~\ref{lem:sepExtGlobalAnisoPfister}. Then 
for $a_d\in K^\times$ with $w(a_d)$ odd, setting 
$\bma:=(a_d,\dts,a_0)$ with $(a_i)_{d>i\geqslant0}$ as 
introduced above,  
we claim that $q_{\bm a}$ is a $\ko,\bm t\sps$test form
which satisfies 
the requirements of Proposition~\ref{prop:existsTestForm}.
Indeed, $q_{a_1,a_0}$~is anisotropic over $\lo=\lx$, by 
the choice of $a_1,a_0\in\ko$. Hence $q_{\bm\pi,a_1,a_0}$ 
is anisotropic over $\widehat l$, by Proposition~\ref{prop1},~2)
(applied with the natural valuation on $\widehat l$ with value
group $\mathbb{Z}^{d-2}$), 
thus anisotropic over $Kw\subset\widehat l$. In particular, 
since $\bm\pi, a_1,a_0$ is a system of $w$-units,
and $w(a_d)$ is odd, one gets that  
$q_{\bm a}=q_{a_d,\bm\pi,a_1,a_0}$ is anisotropic over
 $K_w$, by Proposition~\ref{prop1},~2). 
\end{proof}
\noindent
C) \ {\bf A strengthening of Proposition~\ref{prop:niceLocalGlobal}}
\vskip5pt

In this section, we prove a strengthening of Proposition 
\ref{prop:niceLocalGlobal} under refined hypotheses.

For an arbitrary field $F$, we let $\Val F$ be the 
Riemann--Zariski space (of equivalence classes of valuations) 
of $F$.
We endow $\Val F$ with the \emph{patch topology},
which is the coarsest topology such that the sets of the form
$ \{ v \in \Val F \mid v(a) \geqslant 0 \}, \ a \in F$
are open and closed. It follows that the sets
$\{ v \in \Val F \mid v(b) > 0 \}$, $\{v\in\Val F\mid v(c)=0\}$ 
are open and closed for all $b, c \in F$.

The patch topology makes $\Val F$ a compact Hausdorff space, see for instance
the discussion in \cite{ZariskiSamuelII}, Ch.\ VI §17, proof of Theorem 40.

\begin{lemma}\label{lem:constructiblyOpen}
In the above notation, let $\hns Fw$ be the $w$-henselization 
at $w\in\Val F$.~One~has: 
\vskip2pt
\itm{25}{
\item[{\rm1)}] Let $E|F$ a finite extension. Then the set 
$\clV_{E|F}\!:=\{w\in\Val F\,|\,E \hbox{ is $F$-embeddable 
in $\hns Fw$}\}$ is open in the patch topology.
\vskip2pt
\item[{\rm2)}] Let $q_{\bm a}$ be a quadratic form over $F$. 
Then the set $\clV_{{\bm a}}\!:=\{w\in\Val F\,|\,q_{\bm a} 
\hbox{ is isotropic over $\hns Fw$}\}$ is open in the patch topology.
}
\end{lemma}
\begin{proof} To 1): Recall that the henselization 
$\hns Fw|F$ is a separable algebraic extension,
hence if $\clV_{E|F}$ is non-empty, $E|F$ is separable.
Let $w\in\clV_{E|F}$. We also write $w$ for the 
(canonical) prolongation of $w$ to $F_w$ and its restriction 
to $E$. By Hilbert decomposition theory, (see e.g.\ 
\cite{HenselianElements}, Thm 1.2), $E=F[\eta]$ with 
$\eta$ satisfying $w(\eta) = w(p'(\eta)) = 0$ and $\eta$ 
having minimal polynomial $p(t)=t^n+\sum_{i<n} a_it^i \in F[t]$
such that $w(a_i)\geqslant0$.
Since $\hns Fw|F$ is an immediate extension, there is 
$x \in F$ with $w(x-\eta)>0$, hence $w\big(p(x)\big)>0$, 
and $w\big(p'(x)\big) = 0$. The set
\vskip2pt
\centerline
{$\clV_w\:=\{\tlw\in\Val F\,|\,\tlw(a_i)\geqslant0 \ \hbox{for all} \ i<n,
\tlw(p(x)) > 0, \tlw(p'(x)) = 0\}$}
\vskip2pt
\noindent 
is open (and closed) in the patch topology and $w\in\clV_w$. 
On the other hand, if $\tlw\in\clV_w$, then the polynomial 
$p(t)$ has a zero in the henselization $\hns F{\tlw}$, thus 
$E$ is $F$-embeddable into $\hns F{\tlw}$. Conclude that 
$\clV_w\subset\clV_{E|F}$, hence the latter is open in the 
patch topology, as claimed.
\vskip2pt
To 2): Let $w\in\clV_{{\bm a}}$, that is $q_{\bm a}$ is 
isotropic over $\hns Fw$. Then there is a finite subextension 
$E|F$ of $\hns Fw|F$ such that $q_{\bm a}$ is isotropic
over $E$.
Then $\clV_{\bm a}$ contains the neighborhood 
$\clV_{E|F}$ of $q_{\bm a}$. Thus $\clV_{\bm a}$ is
open.
\end{proof}
\begin{proposition}
\label{prop:testForms} Suppose that $K$ satisfies Hypothesis 
$({\rm H}_d)$. Let $L|K$ be finite separable, and 
$a_d\in K^\times\!$. Suppose that there are a global subfield 
$\ko\subset K$ and $\ko$-algebraically independent 
elements $\bm u=(u_i)_{d>i>1}$ of $K$, such that setting 
%
%
$\,\bm t\!:=(t_i)_i\!:=(u^2_i-u_i)_i$,
there is~a $\ko,\hb1\bmt\,$-\ha1test form $q_{\bm a}$ 
for $a_d$ which is anisotropic over the fields 
$L(\alpha)$ with $\alpha^2-\alpha=a_d/\tttt^2$,
$\theta = (a_{d-1} \dotsm a_1)^N\!$ for all $N > 0$.
Then there exists a prime divisor $w_L$ of $L$
which is trivial on $\ko(\bmt)$ such that $w_L(a_d)>0$ 
is odd, and $q_{\bm a}$ is anisotropic over~$\hns L{{w_L}}$.
\end{proposition}
\begin{proof} First, let $N>0$ be fixed, and for 
$\,\theta=(a_{d-1}\dotsm a_1)^N$ and $\alpha^2-\alpha=a_d/\tttt^2$,
set $\tlK:=L(\alpha)$. Then $q_{\bm a}$ 
is an anisotropic $\ko\sps$nice Pfister form over $\tlK$,
hence Proposition~\ref{prop:niceLocalGlobal} implies that 
there is a non-dyadic prime divisor $\tlw=\tlw_N$ of $\tlK$ such that 
$q_{\bm a}$ is anisotropic over $\hns\tlK\tlw$. 
Recalling
that $\bm a=\big(a_d,(a_i)_{d>i>1}, a_1,a_0\big)=
\big(a_d,(t_i-\epsilon_i)_{d>i>1},a_1,a_0\big)$, we claim:
\vskip5pt
{\bf Claim 1}.  {\it One has $\tlw(a_i)\geqslant0$ for $i<d$.\/}
\vskip2pt 
{\it Proof of Claim 1.\/}
Let $v:=\tlw|_{\ko}$ be the restriction of $\tlw$ to $\ko$.
First, suppose that $v$ is non-trivial. Then $\kovh\subset\hns\tlK\tlw$, 
hence the fact that $q_{\bm a}$ is anisotropic over 
$\hns\tlK\tlw$ implies that $q_{a_1\!,a_0}$ is anisotropic 
over $\kovh$. Since $q_{a_1\!,a_0}$ is $\ko\sps$nice and 
anisotropic over $\kovh$,~Proposition~\ref{prop1},~3) 
applied to $F=\kovh$ and $q_{a_1,a_0}$ implies that
$a_1$ and $a_0$ cannot both be $v$-units, and so
~Proposition~\ref{prop:niceLocalGlobal},~2)
implies that $v(a_0)=0$ and $v(a_1)>0$. Since 
$q_{\bm a}$ is a $\ko,\hb1\bmt\,$-\ha1test form for $a_d$
and $v(a_1)>0$, one has $v(\epsilon_i)=0$ by definition,
thus $\tlw(\epsilon_i)=v(\epsilon_i)=0$  
for $1<i<d$. Second, if $v$ is trivial, then
$\tlw(\epsilon_i)=v(\epsilon_i)=0$ for all $i<d$ as well. 
Hence independently on whether $v$ is trivial or not,
one has $\tlw(a_0)=0$, $\tlw(a_1)\geqslant0$,
and $\tlw(\epsilon_i)=0$ for all $1<i<d$. Next, by contradiction, 
suppose that $\tlw(a_i)<0$ for some $i<d$. Then $1<i<d$, and since 
$a_i=t_i-\epsilon_i$, we must have $\tlw(t_i)<0$. Hence 
$t_i=u_i^2-u_i$ in $K$ implies that $\tlw(u_i)<0$. Therefore, 
$a_i=u_i^2-u_i-\epsilon_i=u^2_ia'_i$ with $a'_i=1-1/u_i+\epsilon_i/u_i^2$ 
a principal $\tlw$-unit. Hence by Proposition~\ref{prop1},~1) it 
follows that $q_{a'_i, a_0}$ is isotropic over $\hns\tlK\tlw$,
thus so are $q_{a_i,a_0}$ and $q_{\bm a}$ -- contradiction!
\vskip5pt
{\bf Claim 2}.  {\it One has 
$\tlw(a_d)>N\tlw(a_i)$ for $i<d$.\/}
\vskip2pt 
{\it Proof of Claim 2.\/} We first prove that that 
$\tlw(a_d)\geqslant\tlw(\tttt^2)$. By contradiction, suppose that
$\tlw(a_d)<\tlw(\tttt^2)$. Then  $\alpha^2-\alpha=a_d/\tttt^2$ 
in $\tlK$ implies $\tlw(\alpha)<0$; hence $\eta\!:=1-1/\alpha$
is a principal $\tlw$-unit. Thus 
$a_d=(\alpha\tttt)^2(1-1/\alpha)=u^2\eta$ with $u=\alpha\tttt$, and we get a 
contradiction as above in the proof of~Claim~1. Second, 
by Claim 1 one has $\tlw(a_i)\geqslant0$ for all $i<d$,  
and therefore, $\tlw(\tttt)=N\sum_{0<i<d}\tlw(a_i)\geqslant0$.
Hence $\tlw(a_d)\geqslant 2\tlw(\tttt)\geqslant2N\tlw(a_i)$ for 
all $i<d$. 
On the other hand, since $q_{\bm a}$ is anisotropic 
over $\hns\tlK\tlw$, it follows by~Proposition~\ref{prop1},~3)
that $\tlw(a_i)\neq0$ for some $i\leqslant d$,
and for such an $i$, we have
$\tlw(a_i)>0$ because $\tlw(a_i)\geqslant0$ by Claim~1. 
Therefore, $\tlw(a_d)\geqslant2N\tlw(a_i)$ for $i<d$ implies
both $\tlw(a_d)>0$ and $\tlw(a_d)>N\tlw(a_i)$ for $i<d$. 
Claim~2 is proved.
\vskip5pt
Coming back to the proof of Proposition~\ref{prop:testForms},
for each integer $N>0$, let $\clV_{\bm a, N}$ be the set of 
valuations $w$ on $L$ satisfying the conditions:
\vskip3pt
$\ha{10}$ i) $q_{\bm a}$ is anisotropic over the henselization $\hns Lw$. 
\vskip2pt
$\ha{8}$ ii) $w(a_i)\geqslant0$ and $w(a_d)> N w(a_i)$ for all $i<d$.
\vskip3pt
\noindent
We notice that $\clV_{\bm a,N}$ is closed, hence compact,
in the patch topology. Indeed, the set of all $w$ satisfying 
condition~ii) is open and closed by definition. Second, 
the complement of the set of valuations satisfying condition~i) 
is open by Lemma~\ref{lem:constructiblyOpen},~2). Finally,
each $\clV_{\bm a,N}$ is non-empty by Claims~1,~2,
because the valuation $\tlw=\tlw_N$ considered there
lies in $\clV_{\bm a,N}$.  
\vskip2pt
Since $\clV_{\bm a,N+1}\subset\clV_{\bm a,N}$, it follows 
by compactness that $\clV_{\bm a}:=\cap_N\clV_{\bm a,N}$
is non-empty, so let us fix $w_{\bm a} \in \clV_{\bm a}$.
%
%
Then $q_{\bm a}$ is anisotropic 
over $\hns L{{w_{\bm a}}}$, and $w_{\bm a}(a_i)\geqslant0$, 
$w_{\bm a}(a_d)> N w_{\bm a}(a_i)$ for all $N > 0$ and
$i<d$. Set $\eup\!:=\{x\in L\,|\,w_{\bm a}(a_d) \leq N\,w_{\bm a}(x)
\text{ for some } N>0\}$.
Then $\eup\subset\clO_{w_{\bm a}}$ is obviously
a prime ideal such that $a_d\in\eup$, $a_i\notin\eup$ for
all $i<d$. Let $w_L$ be the valuation with valuation ring
$\clO_{w_L}=(\clO_{w_{\bm a}})_{\eu p}$. Then $\eum_{w_L}=\eup$, 
and the following hold:
\vskip4pt
a) One has an inclusion of henselizations
$\hns L{{w_L}}\subset \hns L{{w_{\bm a}}}$, so
$q_{\bm a}$ is anisotropic over $\hns L{{w_L}}$. 
\vskip2pt
b) Since $a_i\notin\eup=\eum_{w_L}$, the
$a_i$ are $w_L$-units for $i<d$.
\vskip5pt
{\bf Claim 3}. {\it $w_L$ is trivial on $\ko(\bmt)$, 
and hence $w_L$ is a prime divisor of $L|\ko(\bmt)$.\/}
\vskip5pt
{\it Proof of Claim 3.\/} We first claim that 
$v:=(w_L)|_{\ko}$ is trivial. By contradiction, 
suppose that $v$ is non-trivial, and let 
$\kov\subset \hns L{{w_L}}$ be the Henselization of 
$\ko$ w.r.t.\ $v$ inside $\hns L{{w_L}}$. Since 
$q_{\bm a}$ is a $\ko,\hb1\bmt\,$-\ha1test 
form for $a_d$ which is anisotropic over $\hns L{{w_L}}$, 
it follows that $q_{a_1,a_0}$ is a $\ko$-nice
form which is anisotropic over 
$\kov$.
Hence $v$ is not dyadic.
On the other hand, since $a_i$, $i<d$ are
$w_L$-units, one has $v(a_i) = w_L(a_i) = 0$ for $i = 0, 1$;
hence by Proposition~\ref{prop:basicProperties},~3)
applied to $q_{a_1,a_0}$ over $\kov$ it follows that 
$q_{a_1,a_0}$ is isotropic over $\kov$, contradiction!
Next suppose, by contradiction, that $w_L$ is not trivial on 
$\ko(\bmt)$. Let $F\subset \hns L{{w_L}}$ be the relative 
algebraic closure of $\ko(\bmt)$ in $\hns L{{w_L}}$, and set 
$w:=(w_L)|_F$, $\bm\veps:=(a_{d-1},\dots,a_1,a_0)$. 
Then $q_{\bm\veps}$ is defined over $F$, and $w$ 
is a non-trivial henselian valuation of $F$ such that 
all entries $a_i$ of ${\bm\veps}$ are $w$-units. 
Further, since $w$ is trivial on $\ko$, it follows that
$w$ is non-dyadic. Finally, since $q_{a_1,a_0}$ is 
isotropic over $\kov$ for all archimedean places of 
$\ko$, it follows that $q_{\bm a}$ is isotropic over
$F_v:=F\kov$ for all archimedean places $v$ of $\ko$.
Proposition~\ref{prop1},~3)
implies that $q_{\bm\veps}$ is isotropic over $F$, hence over 
$\hns L{{w_L}}$, because $F\subset \hns L{{w_L}}$. Since 
$q_{\bm\veps}$ is a Pfister subform of $q_{\bm a}$, 
it follows that $q_{\bm a}$ is isotropic over $\hns L{{w_L}}$, 
contradiction! Claim~3 is proved.
\vskip2pt
It is left to prove that $w_L(a_d)$ is positive and odd.
First, $w_L(a_d)>0$
by the definition of $w_L$. Finally, $w_L(a_d)$ is
odd by Proposition \ref{prop:niceLocalGlobal},~2).
\end{proof}
\section{Uniform definability of the geometric 
            prime divisors of $K$} %
\noindent
In this section we show that geometric prime divisors of 
finitely generated fields are uniformly first-order 
definable. This relies in an
essential way on the consequences of the
cohomological principles presented in the 
previous section, and on the (obvious) fact 
that for an $n$-fold Pfister form $q_{\bm a}$, 
whether that $q_{\bm a}$ is (an)isotropic, or 
universal, over $K$ and/or $\tlK=K[\sqrt{-1}\,]$ 
is expressed by formulae in which the $n$ entries in 
$\bma=(a_n,\dts,a_1)$ are the only free 
variables.
The Kronecker dimension 
$\dim(K)$ can be detected in a first-order way,
see \nmnm{Pop}~\cite{P_EEI} Fact 1.1 (3) and Theorem 1.5 (3).
Further, the relatively algebraically closed 
global subfields $\ko\subset K$ of finitely 
generated fields $K$, and algebraic 
independence over such fields $\ko$ are 
uniformly first-order definable by 
\nmnm{Poonen}~\cite{PoonenUniform} Theorem 1.4. 
%
%
\begin{notations/remarks}
\label{notax}
Let $K$ satisfy Hypothesis~(H$_d$). 
\itm{25}{
\item[1)] For $a_d\in K\tms$ consider:
\vskip2pt
\itm{25}{
\item[a)] relatively algebraically closed global subfields 
$\ko\subset K$.
\vskip2pt
\item[b)] $\ko$-algebraically independent elements 
$\bmu=(u_i)_{d>i>1}$ of $K$.
\vskip2pt
\item[c)] systems $\bm\epsilon=(\epsilon_i)_{d>i>1}$ 
of elements of $\ko^\times$ and $a_1,a_0\in\ko\tms$ such that 
$q_{a_1,a_0}$ is a $\ko\sps$nice Pfister form and all $\epsilon_i$ 
are $v$-units for all finite places $v\in\lvP(\ko)$ satisfying $v(a_1)>0$.   
\vskip2pt
\item[d)] Set $\bmt\!:=(t_i)_{d>i>1}=(u_i^2-u_i)_{d>i>1}
=\bmu^2-\bmu$, and $a_i\!:=t_i-\epsilon_i$ for $1\!<\!i\!<\!d$, and 
consider the resulting $\ko,\bmt\sps$test form $q_{\bm a}$ for $a_d$
defined by $\bm a=(a_d,\dots,a_1,a_0)$. 
}
\vskip2pt
\item[2)] For $\bmt, \bmu$ as above, let $k_{\bm t}=k_{\bm u}$ 
be the relative algebraic closure of $\ko(\bmt)$ in $K$, and
$\clD_{K|k_{\bm t}}$ denote the set of prime divisors $w$ of 
$K|k_{\bm t}$. Then $K=k_{\bm t}(C)$ for a unique 
projective normal $k_{\bm t}$-curve $C$, and 
$w\in\clD_{K|k_{\bm a}}$ are in bijection with the closed 
points $P\in C$ via $\clO_w=\clO_P$. 
\vskip2pt
\item[3)] For $\tttt,\tau \in K$ with $\tttt\neq0$, set 
$K_\tttt\!:=K(\alpha)$ and $K_\tau\!:=K(\beta)$, where 
$\alpha^2\!-\alpha=a_d/\tttt^2\!$ and
$\beta^2\!-\beta=\tau^2\hb1/a_d$. Let
$K_{\tttt,\tau}:=K_\tttt(\beta)=K_\tau(\alpha)=K_\tttt K_\tau$ 
be the compositum of $K_\tttt$ and $K_\tau$ over $K$.
}
\vskip2pt
Finally, for the $\ko,\bmt\sps$test 
form $q_{\bm a}$ for $a_d$ introduced above, define:
\vskip5pt
\itm{25}{
\item[4)] $\eua_{{\bm a}}\!:=\{\tau\in K\,|\,\hbox{$q_{\bm a}$ is anisotropic 
over $K_{\tttt,\tau}$ for all $\tttt\in k_{\bm t}^\times$}\}$, 
$\,\clO_{\bm a}\!:=\{a\in K\,|\,a\cdot\eua_{\bm a}\subset\eua_{\bm a}\}$.
\vskip4pt
\item[5)] $\clV_{{\bm a}}\!:=\{w\in\clD_{K|k_{\bm t}}\,|\,\hbox{$w(a_d)>0$ 
    and $q_{\bm a}$ is anisotropic over $K_w$}\}$, and for $w\in\clV_{\bm a}$, set \[\eua_w:=\{\tau\in K\,|\,w(\tau^2)>w(a_d)\} .\]
  Therefore the valuation ring $\clO_w$ is equal to 
 $\{a\in K\,|\,a\cdot\eua_w\subset\eua_w\}$.
}
\end{notations/remarks}
\begin{theorem}
\label{keythm}
Let $K$ satisfy Hypothesis ~{\rm(H$_d$)}. The following hold:
\vskip2pt
\itm{25}{
\item[{\rm1)}] For $\ko$, $\bmu$, $a_d\in K$ and $q_{\bm a}$
as in Notations/Remarks~\ref{notax} above, 
\vskip3pt
\centerline{$\,\eua_{{\bm a}}=\bigcup_{w\in\clV_{\bm a}}\eua_w$, 
\ $\clO_{\bm a}=\bigcap_{w\in\clV_{\bm a}}\clO_w\,$.}
\vskip4pt
\item[{\rm2)}] For every geometric prime divisor $w$ of $K$, there are
$\ko$, $\bm u$, $a_d\in K$ as in Notation/Remarks~\ref{notax} 
above such that $\clV_{{\bm a}}=\{w\}$, and therefore,
\vskip5pt
\centerline{$\,\clO_w=\{a\in K\,|\, a\cdot\eua_{{\bm a}}\subset\eua_{{\bm a}}\}\,$.}
}
\end{theorem}
\begin{proof} To 1):
Let us first argue that
$\eua_{{\bm a}}=\cup_{w\in\clV_{\bm a}}\eua_w$.
\vskip2pt
``$\subset\,$'': Let $\tau \in \eua_{\bm a}$. Set 
$L:=K_\tau$. Then $q_{\bm a}$ is anisotropic over 
$K_{\tttt,\tau}=K_\tau(\alpha)=L(\alpha)$ for all 
$\tttt\in k_{\bm t}^\times$ and $\alpha^2-\alpha=a_d/\tttt^2$;
thus in particular, for $\tttt=(a_{d-1}\dots a_1)^N$ for all
$N>0$. Hence by Proposition~\ref{prop:testForms}, 
there is a prime divisor $w_L$ of $L$ which is trivial 
on $\ko(\bm t)$, hence on its relative algebraic 
closure $k_{\bm t}$ inside $K$, such that $w_L(a_d)>0$ 
is odd, and $q_{\bm a}$ is anisotropic over the 
henselization $\hns L{{w_L}}$. By contradiction, 
assume that $w_L(\tau^2)\leqslant w_L(a_d)$, 
hence $w_L(\tau^2)< w_L(a_d)$, because $w_L(a_d)$ 
is odd. Then $w_L(\tau^2/a_d)<0$, hence $w_L(\beta)<0$, 
so $a'_d:=1 - 1/\beta$ is a principal $w_L$-unit, thus 
$q_{a'_d,a_0}$ is isotropic over $\hns L{{w_L}}$ by 
Proposition~\ref{prop1},~1). Since 
$a_d=(a_d \beta/\tau)^2(1-1/\beta)$, one has
$q_{a_d,a_0}\approx q_{a'_d,a_0}$ over $\hns L{{w_L}}$, 
hence $q_{a_d, a_0}$ is isotropic over $\hns L{{w_L}}$. 
Thus $q_{\bm a}$ is isotropic over $\hns L{{w_L}}$ as well, 
contradiction!
Therefore $w_L(\tau^2) > w_L(a_d)$.
Setting $w := (w_L)|_K$, we see that $w \in \clV_{\bm a}$
and $\tau \in \eua_w$.
\vskip2pt
``$\,\supset\,$'': Let $w\in\clV_{\bm a}$ and $\tau\in\eua_w$ 
be given, i.e., $w(\tau^2)>w(a_d)$.
Let $\tttt \in k_{\bm t}^\times$ be arbitrary.
By definitions,
$w$ is trivial on $k_{\bm t}$, $w(a_d)>0$, and $q_{\bm a}$ 
is anisotropic over the henselization $\hns Kw$.
As $a_i \in k_{\bm t}$ and therefore $w(a_i) = 0$ for $i < d$,
by Proposition~\ref{prop:niceLocalGlobal},~2) it follows that $w(a_d)$ is odd.
Therefore  
one has 
$w(a_d/\tttt^2) = w(a_d) > 0$ and $w(\tau^2/a_d) > 0$.
Hence if $\alpha^2-\alpha=a_d/\tttt^2$
and $\beta^2-\beta=\tau^2/a_d$, then $\alpha,\beta\in \hns Kw$
by Hensel's Lemma. Thus 
$K_{\tttt,\tau}\subset \hns Kw$, and this implies that 
$q_{\bm a}$ is anisotropic over $K_{\tttt,\tau}$.
Therefore, $\tau\in \eua_{\bm a}$.
\vskip2pt
We have shown that 
$\eua_{{\bm a}}=\bigcup_{w\in\clV_{\bm a}}\eua_w$.
It follows immediately that
$\clO_{\bm a} \supset \bigcap_{w \in \clV_{\bm a}} \clO_w$.
For the other inclusion, let $w \in \clV_{\bm a}$ 
and set $\mu_w:= \min \{ w(y') \,|\, y' \in \eua_w \}$. Here
the minimum exists since $\eua_w \subseteq \clO_w$.
For $x \in K \setminus \clO_w$
set
\[ \Sigma_{w,x}\!:=\{y \in \eua_w\,|\, w(y)=\mu_w,
  w'\big((xy)^2\big) < w'(a_d) \ \forall\ w' \in \clV_{\bm a} \setminus \{ w \} \}.
\]
\indent
Since $\clV_{\bm a} \subset \clD_{K|k_{\bm a}}$ is
finite, the set $\Sigma_{w,x}$ is non-empty by
weak approximation (it is defined by an open condition
for every $w' \in \clV_{\bm a}$ including $w$).
Let $y_0\in\Sigma_{w,x}$.
Then $y_0 \in \eua_w \subseteq \eua_{\bm a}$, but
$xy_0 \not\in \eua_w$ by minimality of $w(y_0)$ since $w(x) < 0$,
and 
$xy_0 \not\in \bigcup_{w' \in \clV_{\bm a} \setminus \{ w \}} \eua_{w'}$
by definition of $\Sigma_{w,x}$.
Hence $xy_0 \not \in \eua_{\bm a}$,
and thus $x \cdot \eua_{\bm a} \not\subset \eua_{\bm a}$.
This shows $\clO_{\bm a} \subset \clO_w$ for all $w$,
and therefore $\clO_{\bm a} = \bigcap_{w \in \clV_{\bm a}} \clO_w$.
\vskip2pt
To 2): Let $w$ be a geometric prime divisor of $K$.
Then by Lemma~\ref{lmm2}, there is a (maximal) 
global subfield $\ko\subset K$ and $\bmu=(u_i)_{d>i>1}$ 
algebraically independent over $\ko$ such
that $w$ is trivial on $\ko(\bmu)$, and $\Kw|\ko(\bmu)$ 
is finite separable. Set $\bmt:=\bmu^2-\bmu$. 
Then $\ko(\bm u)|\ko(\bm t)$ is a finite abelian extension,
hence $\Kw|\ko(\bmt)$ is finite separable, and 
$k_{\bm t}=k_{\bm u}$ inside $K$. Further recall
that $K=k_{\bm t}(C)$ for a (unique) projective normal 
$k_{\bm t}$-curve $C$,
and there is a unique closed point $P \in C$ with local ring $\clO_P = \clO_w$.
By Riemann--Roch for the projective normal $k_{\bm t}$-curve $C$,
for every sufficiently large $m \gg 0$, there is
a function $f \in k_{\bm t}(C)^\times$ with $(f)_\infty = m P$.
Let us fix such $m \gg 0$ which is odd, and such $f$.
Then the element $a_d := 1/f$ of the function field $K = k_{\bm t}(C)$
has $P \in C$ as its unique zero, and $w(a_d) = m$.

Applying Proposition~\ref{prop:existsTestForm}, we find 
$\bm\epsilon=(\epsilon_i)_{d>i>1}\in{\ko^\times}^{d-2}$ 
and $a_1,a_0\in\ko^\times$ such that setting $\bma=(a_d,\,\dts,a_0)$
with $a_i=t_i-\epsilon_i$, $1<i<d$, the resulting $q_{\bm a}$ 
is a $\ko,\bmt\,$-\ha1test 
form for $a_d$ which is anisotropic over $K_w$. Moreover,
since $w$ is the unique prime divisor of $K|k_{\bm t}$
with $w(a_d)>0$, it follows that $\clV_{{\bm a}}=\{w\}$.
Hence by assertion~1) above, 
$\clO_w=\{a\in K\,|\, a\cdot\eua_{{\bm a}}\subset\eua_{{\bm a}}\}\,$.
\end{proof}
\begin{recipe} 
\label{therecipe}
One gets a uniform first-order description of the
valuation rings $\clO_w$ of all the geometric prime 
divisors $w$ of $K$ along the following steps:
\vskip4pt
\itm{25}{
\item[1)] Consider the uniformly first-order 
definable $\ko$, $\bmu=(u_i)_{d>i>1}$, 
$\ko\!\subset\!k_{\bm t}\subset\!K$, and further
$\bma:=(a_d,\dots,a_1,a_0)$ and $q_{\bm a}$ 
as in Notations/Remarks~\ref{notax}.
\vskip2pt
\item[2)] Check whether $\clO_{\bm a}$ as defined 
above is a non-trivial valuation
ring of $K$. If so, $\clO_{\bm a}$ is a geometric prime 
divisor of $K|k_{\bm t}$ by Theorem \ref{keythm},~1).
\vskip3pt
\item[3)] By Theorem \ref{keythm},~2), the valuation ring
$\clO_w$ of any geometric prime divisor $w$ of $K$ arises
as above.
}
\end{recipe}

\noindent
This concludes the proof of Theorem~\ref{thm2}.
\begin{remark}\label{rem:defPrimeDivisorsRumelyPop}
Theorem~\ref{thm2} was stated and proved for 
finitely generated fields $K$ with $d=\dim(K)>2$. 
As we now explain, for finitely generated fields 
$K$ of Kronecker dimension $d=1,2$, there are 
formulas $\val_1$ and $\val_2$ which uniformly describe 
the prime divisors in case $d=1$, respectively the 
geometric prime divisors in case $d=2$.~For $d=1$ (i.e., for global fields),
all prime divisors are uniformly definable
by \nmnm{Rumely} \cite{Rumely}, Introduction,~I.
The prime divisors are geometric if and only if $K$ is
a global function field, which is a definable condition
by II loc.\ cit. For $d=2$, uniform definability of geometric 
prime divisors is one of the main results of \nmnm{Pop}
\cite{P_EEI2}: Use that for every geometric prime divisor
$v$ of $K$ we can find a global subfield $k_1 \subseteq K$
with $v$ trivial on $k_1$ such that $K$ is the function
field of a smooth curve over $k_1$, and then apply
\cite{P_EEI2} Theorem 1.2 (cf.~Conclusion 5.2).
\end{remark}

\section{Proof of the Main Theorem}

We will now prove that every field satisfying Hypothesis 
$(\HH_d)$ is bi-interpretable with the ring $\Z$, 
building on the uniform definability of the geometric 
prime divisors. The insight that this is possible is 
due to \nmnm{Scanlon}~\cite{Sc} (more precisely 
one can use \cite[Thm 4.1]{Sc}, because the part 
of the proof needed here is not affected by the gap 
in the recipe of the definability of prime divisors 
in that paper). For the convenience of the reader, 
we instead build on the later \cite{AKNS}, where 
the bi-interpretability result is established for finitely 
generated integral domains (as well as some other rings). 
\begin{proposition}\label{prop:goodDefinableSubring}
Let $K$ satisfy Hypothesis $(\HH_d)$, $\clT$ denote a 
transcendence basis of~$K$, and $R_{\clT}$ be the integral 
closure in $K$ of the subring generated by $\clT$. Then the 
ring $R_{\clT}$ is a finitely generated domain which is
first-order definable (with parameters).
\end{proposition}
\begin{proof} Let $\kappa\subset K$ be the constant field of $K$. 
By \nmnm{Poonen}~\cite{PoonenUniform} Theorem 1.3, $\kappa$ 
is first-order definable.
In characteristic zero, i.e.\ if $\kappa$ is a number field, 
by \nmnm{Rumely}~\cite{Rumely}, Introduction, III, the ring 
of integers $\clO_\kappa$ is first-order definable. To 
fix notation, we set $A:=\kappa$ if $\chr(K)>0$ and 
$A:=\clO_\kappa$ otherwise. Hence $A\subset K$
is first-order definable, and $R:=R_\clT$ is 
the integral closure of $A[\clT]$ in the 
field extension $K|K_0$, where $K_0:=\kappa(\clT)$. 
Further, $R$ is a finite $A[\clT]$-module (see e.g.\ 
\cite{Eisenbud_CommAlg},  Corollary\ha213.13, and 
Prop.~13.14), hence $R$ is a finitely generated 
ring. Hence it is left to prove that $R=R_\clT$
is first-order definable.
\vskip2pt 
Let $S=S_\clT\subset K$ be the integral closure of 
$\kappa[\clT]$ in $K$, and $\clW_\clT$ be the set 
of geometric prime divisors $w$ of $K$ such that 
$\clT\subset\clO_w$. Since the geometric prime divisors 
of finitely generated fields $K$ with $\dim(K)=d$ are 
a first-order definable family (by Theorem \ref{thm2})
it follows that $\clW_\clT$ is a
first-order definable family.
\vskip2pt
We claim that $S=\bigcap_{w\in\clW_\clT}\clO_w$. 
First, ``$\,\subset\,$'' is clear, because 
$\clT\subset\clO_w$ implies that $S\subset\clO_w$, 
hence  $S\subset\bigcap_{w\in\clW_\clT}\clO_w$. Second, 
for ``$\,\supset\,$'' let $\clX^1\subset\Spec(S)$ 
be the set of minimal non-zero prime ideals $\eup$. 
Then the local rings $S_\eup$, $\eup\in\clX^1$ are 
valuation rings of geometric prime divisors of $K$,
and $S=\bigcap_{\eup\in\clX^1}S_\eup$, see e.g.\ 
\cite{Ma},~Thm~11.5,~(ii). Hence
$S=\bigcap_{\eup\in\clX^1}S_\eup\supset\bigcap_{w\in\clW_\clT}\clO_w$. 
\vskip2pt
In particular, 
the ring $S=\bigcap_{w\in\clW_\clT}\clO_w$ is a
definable subset of $K$.
\vskip5pt
\underbar{Case 1}. $\chr(K)>0$. Then $A=\kappa$ 
is a finite field, hence $R_\clT=S_\clT$ is 
first-order definable, and there is nothing left to prove.
\vskip5pt
\underbar{Case 2}. $\chr(K)=0$. Set $e=\td(K|\kappa)$. 
The {\it geometric prime $e$-divisors\/} of $K$ are the 
valuations $\euw$ of $K$ which are trivial on $\kappa$ 
and have $\euw K=\lvZ^e$ lexicographically ordered.
By general valuation theory, a valuation $\euw$ of $K$
is a geometric prime $e$-divisor of $K$ if and only if $\euw$ 
is of the form $\euw=w_1\circ\dots\circ w_e$ (as
composition of places) such that $w_e$ is a discrete
valuation of $K$, and $w_i$ is a discrete valuation
of the residue field $\kappa(w_{i+1})$ of $w_{i+1}$ for $i<e$.
Since $\dim K = e + \dim \kappa$, each $w_i$ must in fact
be a geometric prime divisor of $\kappa(w_{i+1})$.
\vskip2pt
By uniform definability of geometric prime divisors
of fields of fixed finite Kronecker dimension
(Theorem \ref{thm2} and Remark
\ref{rem:defPrimeDivisorsRumelyPop}), the set 
$\clD^e_{K|\kappa}$ of geometric prime $e$-divisors
is a first-order definable family, using induction on
Kronecker dimension and the following easy observation: 
\begin{fact}
If $\clO_{w'}\subset F$ and $\clO_{w''}\subset \Fw'$ 
are first-order definable valuation rings, then the
residue map $\clO_{w'}\to \Fw'$ is first-order definable,
hence so is $\clO_{w''\circ\, w'}\subset F$, as being 
the preimage of the first-order definable set $\clO_{w''}$ 
under the first-order definable map $\clO_{w'}\to \Fw'$.
\end{fact}
Further, the residue fields $\kappa_\euw:=K\!\euw$ 
are finite extensions of $\kappa$, hence 
$\lvP_{\rm fin}(\kappa_\euw)$ and the integral closures 
$A_\euw|A$ of $A$ in $\kappa_\euw$ are uniformly first-order
definable, see~\cite{Rumely}, Introduction,~I,~II,~III.
For $\euw\in\clD^e_{K|\kappa}$ 
and a prime divisor $v\in\lvP_{\rm fin}(\kappa_\euw)$, 
we set $\euw_v:=v\circ\euw$, and for the given transcendence 
basis $\clT=(t_1,\dots,t_e)$ of $K|\kappa$, denote:
\[
\clV_\clT=\{\,\euw_v\mid \euw\in\clW_\clT, 
v\in\lvP_{\rm fin}(\kappa_\euw) \ \hbox{such that } \
\euw_v(t_i)\geqslant0 \ \hbox{ for } i=1,\dots,e\,\}.
\]
Note that $\clV_\clT$ is a definable family by the fact that $\clW_\clT$ and 
$\lvP_{\rm fin}(\kappa_w)$ are so. Hence the
definability of $R_\clT$ follows from
Lemma~\ref{lemma-1} below.
\end{proof}
\begin{lemma}
\label{lemma-1}
One has $R_\clT= \bigcap_{\euw_v\in\clV_\clT}\,\clO_{\euw_v}$. 
Hence $R_\clT$ is first-order definable.
\end{lemma} 
\begin{proof} 
For every $\euw_v=v\circ\euw\in\clV_\clT$, 
one has $\clO_{\euw_v}\subset\clO_\euw$. Hence 
setting $R'_\clT\!:=\bigcap_{\euw_v\in\clV_\clT}\clO_{\euw_v}$
and reasoning as above in the case of $S_\clT$, one gets 
$R_\clT\subset R'_\clT\subset S_\clT$. Hence to complete
the proof of Lemma~\ref{lemma-1}, it is left to prove the 
converse inclusion $R_\clT\supset R'_\clT$.
\vskip2pt
First, setting $K_0\!:=\kappa(\clT)$, one has that $K|K_0$ 
is a finite field extension, and $R_\clT\subset S_\clT$ are 
the integral closures of $R_{0,\clT}\!:=A[\clT]\subset
\kappa[\clT]=:S_{0,\clT}$ in the field extension $K|K_0$. 
Define $\clW_{0,\clT}$ and $\clV_{0,\clT}$ correspondingly 
for $K_0$ instead of $K$, and notice that $\clW_\clT$ 
and $\clV_\clT$ are 
%
%
the prolongations of 
$\clW_{0,\clT}$
%
%
and $\clV_{0,\clT}$ to $K$
under the finite field extension $K|K_0$. Then by
the characterization of integral closure using valuations,
$R'_\clT$ is the integral closure of 
$R'_{0,\clT}:=\bigcap_{\euw_v\in\clV_{0,\clT}}\clO_{\euw_v}$,
in the field extension $K|K_0$. Therefore, it is sufficient
to prove that $R_{0,\clT}=R'_{0,\clT}$, or equivalently,
to prove Lemma~\ref{lemma-1} in the special case
$K=K_0=\kappa(\clT)$, $R_\clT=R_{0,\clT}=A[\clT]$,
and that will be assumed from now on. 
\vskip2pt
We already proved that $A[\clT]=R_\clT$ is contained
in $R'_\clT$, hence it is left to prove that $R'_\clT\subset A[\clT]$.
Recalling that $R'_\clT\subset S_\clT=\kappa[\clT]$,
and $A[\clT]=\bigcap_{v\in\lvP_{\rm fin}(\kappa)}\clO_v[\clT]$,
we have to prove: 
\vskip5pt
{\bf Claim.} {\it Every $f\in R'_\clT$ is in $\clO_v[\clT]$
for all $v\in\lvP_{\rm fin}(\kappa)$.\/}
\vskip5pt 
{\it Proof of Claim.\/} Let $f\in R'_\clT$ be given,
and $v\in\lvP_{\rm fin}(\kappa)$ be fixed, say with
residue field $\kappa_v=\kappa v$. Since 
$R'_\clT\subset\kappa[\clT]$, we can set $f = c \cdot g$ 
with $c \in \kappa$ and $g \in\clO_v[\clT]$ such that 
the reduction $\overline g \in\kappa_v[\clT]$ 
is non-zero,
e.g.\ $c=0$ and $g=1$ if $f=0$. Hence in order 
to prove the Claim, it is sufficient to prove
that $v(c)\geqslant0$.
Since $\overline g\neq0$, there is an $e$-tuple 
$\bm\zeta$ in the algebraic closure of $\kappa_v$ 
such that $\oli g(\bm\zeta) \neq 0$. Then $\bm\zeta$
is an $e$-tuple of roots of unity of order prime 
to $\chr(\kappa_v)$, and we identify $\bm\zeta$ 
with its lift in the algebraic closure of $\kappa$.
Let $\euw\in\clW_\clT$ be such that 
$\clT\mapsto\bm\zeta$ under $\clO_\euw\to K\euw$.
Then $K\euw=\kappa[\bm\zeta]=:\!\kappa'\!$,
and if $v'$ prolongs $v$ to $\kappa'$, then the valuation
$\euw_{v'}:=v'\circ\euw$ lies in $\clV_\clT$ and satisfies:
$g\mapsto g(\bm\zeta)\mapsto\oli g(\bm\zeta)\neq0$ 
under $\clO_{\euw_{v'}}\to\clO_{v'}\to\kappa'v'=K_0\euw_{v'}$. 
Hence $g$ is a $\euw_{v'}$-unit, implying that
$\euw_{v'}(f)=\euw_{v'}(c)$. 
Finally,
since $f\in R'_\clT
\subset\clO_{\euw_{v'}}$,
one~has $\euw_{v'}(f)\geqslant0$, hence
$v(c)=v'(c)=\euw_{v'}(c)=\euw_{v'}(f)\geqslant0$, 
concluding that $v(c)\geqslant0$, thus 
$f=c\cdot g\in\clO_v[\clT]$,~as~claimed.
\end{proof}
\begin{remark}
The first-order definition from the proof of 
Proposition~\ref{prop:goodDefinableSubring} can be 
seen to be uniform for fixed $d$, i.e.~allowing for 
variables for the elements of $\clT$, the defining 
formula can be chosen not to vary for all fields $K$ 
satisfying Hypothesis $(\HH_d)$.
\end{remark}
We are now ready to prove the bi-interpretability theorem: 
a field $K$ satisfying Hypothesis $(\HH_d)$ is bi-interpretable 
with $\Z$, where both $K$ and $\Z$ are considered as 
structures in the language of rings. We refer the reader 
to \cite[Section 2]{AKNS} for a brief introduction to the
notion of bi-interpretability.
\begin{proof}[Proof of the bi-interpretability theorem]
Let $K$ be a field satisfying $(\HH_d)$, and 
$R_\clT \subseteq K$ the definable subring from 
Proposition \ref{prop:goodDefinableSubring}. Since 
$R=R_\clT$ is a finitely generated integral domain, 
it is bi-interpretable with the ring $\Z$ by \cite[Thm~3.1]{AKNS}.

The field $K$ is interpretable in $R$ as a localization, 
cf. \cite[Examples 2.9 (4)]{AKNS}.  Then $K$ is definably 
isomorphic to the interpreted copy of $K$ in the definable 
subset $R \subseteq K$, namely by assigning to each 
$x \in K$ the class of pairs $(a, b) \in R \times (R \setminus \{ 0 \})$ 
with $x = a/b$, and likewise $R$ is definably isomorphic 
to the copy of $R$ defined in the interpreted copy of 
$K$, namely by identifying $r \in R$ with the pair 
$(r, 1)$ (thought of as standing for $\frac{r}{1}$ in 
$\operatorname{Frac}(R) = K$). Thus $K$ is  
bi-interpretable with $R$, and therefore, by transitivity, 
bi-interpretable with $\Z$.
\end{proof}

The resolution of the strong form of the EEIP now follows 
from \cite[Proposition 2.28]{AKNS}.
$\ha0$
\vskip2pt
$\ha0$

\end{document}